\documentclass{amsart}
\usepackage{amsmath, amssymb}
\usepackage{color,graphicx}
\input xy
\xyoption{all}
\usepackage{enumerate}

\newcommand{\Aut}{\ensuremath{\operatorname{Aut}}}

\newcommand{\N}{\mathbb{N}}

\newcommand{\G}{\Gamma}
\newcommand{\VG}{V(\G)}
\newcommand{\EG}{E(\G)}
\newcommand{\Gvn}{G_v^{[n]}}
\newcommand{\XkL}{\mathcal{X}(k,L)}

\newcommand{\la}{\langle}
\newcommand{\ra}{\rangle}

\DeclareMathOperator{\Wr}{wr}

\newcommand{\CAT}{\operatorname{CAT}}
\newcommand{\girth}{\operatorname{girth}}
\newcommand{\st}{\operatorname{st}}

\newcommand{\SL}{\ensuremath{\operatorname{SL}}}
\newcommand{\Sym}{\ensuremath{\operatorname{Sym}}}

\newcommand{\soc}{\ensuremath{\operatorname{soc}}}

\newcommand{\GL}{\ensuremath{\operatorname{GL}}}
\newcommand{\PSL}{\ensuremath{\operatorname{PSL}}}
\newcommand{\PSp}{\ensuremath{\operatorname{PSp}}}
\newcommand{\PSU}{\ensuremath{\operatorname{PSU}}}
\newcommand{\PGL}{\ensuremath{\operatorname{PGL}}}
\newcommand{\PG}{\ensuremath{\operatorname{PG}}}
\newcommand{\PGaL}{\operatorname{P\Gamma L}}
\newcommand{\PGaSp}{\operatorname{P\Gamma Sp}}
\newcommand{\PGammaU}{\operatorname{P\Gamma U}}
\newcommand{\GF}{\ensuremath{\operatorname{GF}}}

\def\GammaL{{\rm \Gamma L}}

\def\mC{\mathbb{C}}
\def\mF{\mathbb{F}}
\def\mQ{\mathbb{Q}}

\def\polhk#1{\setbox0=\hbox{#1}{\ooalign{\hidewidth
    \lower1.0ex\hbox{$\,\lhook$}\hidewidth\crcr\unhbox0}}}
\newcommand{\Swiatkowski}{\'Swi{\polhk{a}}tkowski}

\newtheorem{theorem}{Theorem}[section]
\newtheorem{prop}[theorem]{Proposition}
\newtheorem{lemma}[theorem]{Lemma}
\newtheorem{corollary}[theorem]{Corollary}

\newtheorem{definition}[theorem]{Definition}

\theoremstyle{definition}
\newtheorem{example}[theorem]{Example}

\title{Characterising star-transitive and st(edge)-transitive graphs}

\author[Giudici]{Michael Giudici}
\email{michael.giudici@uwa.edu.au}

\author[Li]{Cai Heng Li}
\email{cai.heng.li@uwa.edu.au}

\author[Seress]{\'Akos Seress}
\email{akos@math.ohio-state.edu}

\author[Thomas]{Anne Thomas}
\email{anne.thomas@sydney.edu.au}

\address[Giudici, Li and Seress]{School of Mathematics and Statistics, The University of Western Australia, 35 Stirling Highway, Crawley WA 6009, Australia}
\address[Seress]{The Ohio State University, Department of Mathematics, 231 W 18th Avenue, Columbus, OH 43210, USA}
\address[Thomas]{School of Mathematics and Statistics F07, University of Sydney NSW 2006, Australia}

\thanks{This research was supported in part by ARC Grants DP120100446 (for Giudici), DP1096525 (for Li and Seress), and DP110100440 (for Thomas). \'Akos Seress is also supported in part by an Australian Professorial Fellowship and by the NSF. Anne Thomas is supported in part by an Australian Postdoctoral Fellowship.}

\date{\today}

\begin{document}

\begin{abstract}  
Recent work of Lazarovich provides necessary and sufficient conditions on a graph $L$ for there to exist a unique simply-connected $(k,L)$-complex.  The two conditions are symmetry properties of the graph, namely star-transitivity and st(edge)-transitivity. In this paper we investigate star-transitive and st(edge)-transitive graphs by studying the structure of the vertex and edge stabilisers of such graphs. We also provide new examples of graphs that are both star-transitive and st(edge)-transitive.
\end{abstract}

\maketitle

\section{Introduction}
\label{s:introduction}

In this paper we investigate graph-theoretic conditions on the links of vertices in certain simply-connected polygonal complexes, such that by recent work of Lazarovich \cite{L}, local data specify these complexes uniquely.

A $2$-dimensional CW-complex $X$ is called a \emph{polygonal complex} if both of the following hold:
\begin{enumerate}
\item the attaching maps of $X$ are homeomorphisms; and
\item the intersection of any two closed cells of $X$ is either empty or exactly one closed cell.
\end{enumerate}
These conditions imply that the $1$-skeleton of $X$, that is, the graph given by the vertices and edges of $X$, is simple, that is, has no loops or multiple edges.  More details on these and subsequent topological terms will be provided in Section \ref{s:algtop}. We will refer to the closed $1$-cells of $X$ as \emph{edges} and the closed $2$-cells as \emph{faces}.  The boundary of any face $P$ of $X$ is a cycle of $k$ edges, for some integer $k \geq 3$, and so we may also refer to $P$ as a \emph{$k$-gon}.  For each  vertex $x$ of $X$, the \emph{link} of $x$ is the simple graph with vertex set the edges of $X$ containing $x$, edge set the faces of $X$ containing $x$, and two vertices of the link adjacent if and only if the corresponding edges in $X$ are contained in a common face.  We do not usually think of polygonal complexes as being embedded in any space.   Polygonal complexes are, however, often metrised so that each face is a regular Euclidean polygon of side length one.  In some cases the faces may be metrised as hyperbolic polygons instead.  

Polygonal complexes play an important role in combinatorial and geometric group theory.  The Cayley $2$-complex of a group presentation is a polygonal complex, and many groups are investigated by considering their action on an associated polygonal complex of nonpositive or negative curvature (see Section \ref{s:algtop} for the definitions of these curvature conditions).  From a slightly different point of view, if $X$ is a simply-connected, locally finite polygonal complex, then the automorphism group of $X$ is naturally a totally disconnected locally compact group.  As explained in the survey \cite{FHT}, we may hope to extend to this setting results from the theory of Lie groups and their lattices.  

Given an integer $k \geq 3$ and a simple graph $L$, a \emph{$(k,L)$-complex} is a polygonal complex $X$ such that each face is a $k$-gon and the link of each vertex is $L$.    For example, if $C_n$ denotes the cycle graph on $n$ vertices, then the boundary of a tetrahedron is a $(3,C_3)$-complex, and the regular tiling of the Euclidean plane with squares is a $(4,C_4)$-complex.   The (cartesian) product of two trees of bivalency $\{m,n\}$ is a $(4,K_{m,n})$-complex, where $K_{m,n}$ is the complete bipartite graph on $m + n$ vertices.  If $L$ is the Heawood graph, that is, the point-line incidence graph of the projective plane $\mathrm{PG}(2,2)$, then the buildings for both $\SL_3(\mQ_2)$ and $\SL_3(\mF_2((t)))$ are $(3,L)$-complexes, which are not isomorphic.  There is a close relationship between $(k,L)$-complexes and rank $3$ incidence geometries, which we discuss  in Section \ref{s:algtop}.

We say that the pair $(k,L)$ satisfies the \emph{Gromov Link Condition} if $k \geq m$ and $\girth(L) \geq n$, where $(m,n) \in \{ (3,6), (4,4), (6,3) \}$.  If this condition holds, then as we recall in Section \ref{s:algtop}, a simply-connected $(k,L)$-complex $X$ is either nonpositively or negatively curved, as its faces are metrised as Euclidean or hyperbolic polygons respectively.  We assume throughout this paper that the graph $L$ is finite and connected, equivalently $X$ is locally finite and has no local cut-points.

Denote by $\XkL$ the collection of all simply-connected $(k,L)$-complexes (up to isomorphism).  As described in \cite{FHT}, simply-connected $(k,L)$-complexes are often constructed as universal covers of finite $(k,L)$-complexes.  Also, discrete groups which act upon simply-connected $(k,L)$-complexes are often constructed either as fundamental groups of finite $(k,L)$-complexes,  or as fundamental groups of complexes of finite groups over polygonal complexes which have all faces $k$-gons and the link of the local development at each vertex being the graph $L$.  In each of these cases, in order to identify the universal cover, which will be \emph{some} simply-connected $(k,L)$-complex, it is essential to know whether $|\XkL| = 1$.

The question of when $|\XkL| = 1$ has been addressed by several authors, with the most complete answer to date given by Lazarovich~\cite{L}.    Ballmann and Brin \cite{BB} provided an inductive construction of simply-connected $(k,L)$-complexes whenever the pair $(k,L)$ satisfies the Gromov Link Condition and certain obvious obstructions do not occur.  (For example, if $k$ is odd and $L$ is bipartite of bivalency $\{\ell,r\}$ with $\ell \neq r$, then there is no $(k,L)$-complex.)  Moreover, Ballmann and Brin and independently Haglund \cite{H2} showed that there are uncountably many non-isomorphic simply-connected $(k,K_n)$-complexes when $k\geq 6$ and $n\geq 4$ (here, $K_n$ is the complete graph on $n$ vertices).   We discuss other results prior to \cite{L}, on both uniqueness and non-uniqueness of $(k,L)$-complexes, in Section \ref{s:examples}.  

In \cite{L}, the graph-theoretic notions of \emph{star-transitivity} and \emph{st(edge)-transitivity} are introduced.  We recall the definition of these terms in Definitions~\ref{def:star and iso} and \ref{def:star trans} below, where we also define \emph{$G$-star-transitivity} and \emph{$G$-st(edge)-transitivity} for $G$ a subgroup of $\Aut(L)$.  The Uniqueness Theorem of \cite{L}  states that if $k \geq 4$, the pair $(k,L)$ satisfies the Gromov Link Condition and $L$ is both star-transitive and st(edge)-transitive, then $|\XkL| \leq 1$.  Moreover, Theorem A of \cite{L} states that for $k \geq 4$ even and $(k,L)$ satisfying the Gromov Link Condition, $|\XkL| = 1$ if and only if $L$ is both star-transitive and st(edge)-transitive.  

In this paper we investigate star-transitivity and st(edge)-transitivity for graphs, expanding on the results and examples given in \cite[Section 1.1]{L}.  Our main results are the following. 
(For the group notation and graph definitions used in these theorems, we refer to Section~\ref{s:definitions}.)
We first consider vertex-transitive graphs, that is, graphs $L$ such that $\Aut(L)$ acts transitively on the set of vertices of $L$.

\begin{theorem}\label{v-star-st(edge)-t}
Let $L$ be a connected graph of valency $r\ge3$, let $G\le\Aut(L)$ be vertex-transitive and let $v$ be an arbitrary vertex of $L$.  Then $L$ is $G$-star-transitive and $G$-st(edge)-transitive if and only if
one of the following holds:
\begin{enumerate}[$(1)$]
\item $L$ is $(G,3)$-transitive, and $G_v=S_r\times S_{r-1}$.

\item $L$ is cubic and $(G,4)$-arc-transitive, and $G_v=S_4$ or $S_4\times S_2$.

\item $L$ is of valency $4$ and $(G,4)$-arc-transitive, and $G_v=3^2{:}\GL(2,3)$ or $[3^5]{:}\GL(2,3)$.
\end{enumerate}
\end{theorem}

We note that all cases in Theorem \ref{v-star-st(edge)-t} give rise to examples. Case (1) is realised by the Odd graphs (see Section \ref{s:odd}), Case (2) with $G_v=S_4$ is realised by the Heawood graph and Case (2) with $G_v=S_4\times S_2$ is realised  by the generalised quadrangle associated with the symplectic group $\PSp(4,2)$. The groups $3^2{:}\GL(2,3)$ or $[3^5]{:}\GL(2,3)$ in part (3) of Theorem \ref{v-star-st(edge)-t} are the parabolic subgroups of $\PGL(3,3)$ and the exceptional group of Lie type $G_2(3)$. Case (3) is then realised by the point-line incidence graph of the projective plane $\PG(2,3)$ and the generalised hexagon associated with $G_2(3)$.

We then consider graphs which are not vertex transitive.

\begin{theorem}\label{vtx-intrans}
Let $L$ be a connected graph with minimal valency at least three, let $G  \le \Aut(L)$, and let $\{ v,w \}$ be an arbitrary edge of $L$. Assume that $L$ is $G$-star-transitive and $G$-st(edge)-transitive but not vertex-transitive. Then $L$ is a locally $(G,3)$-arc-transitive bipartite graph of bi-valency $\{ \ell,r\}$ with $\ell\not=r$, and after possibly exchanging the role of $v$ and $w$,
one of the following holds:
\begin{enumerate}[$(1)$]
\item $G_{vw}^{[1]}$ acts nontrivially on both $\Gamma_2(v)$ and $\Gamma_2(w)$, and $\{ \ell,r \}=\{ 3,5 \}$;

\item $G_{vw}^{[1]}=1$, $G_v=S_r\times S_{\ell-1}$, and $G_w=S_\ell\times S_{r-1}$;

\item $G_w^{[2]}=1$, and 
$(A_{r-1})^{\ell-1}\le G_{vw}^{[1]}\le (S_{r-1})^{\ell-1}$, 

\[\begin{array}{rllll}
(A_r\times (A_{r-1})^{\ell-1}).2.S_{\ell-1}&\leqslant &G_v&\leqslant&S_r\times (S_{r-1}\Wr S_{\ell-1}), \mbox{and}\\
(A_{r-1})^{\ell}.2.S_\ell& \leqslant& G_w &\leqslant &S_{r-1}\Wr S_\ell;
\end{array}\]

\item  $|\Gamma(v)|=r\le 5$.
\end{enumerate}
\end{theorem}

We provide examples in each of the first three cases. Case (1) is realised by the generalised quadrangle associated with the unitary group $\PSU(4,2)$ (Example \ref{eg:hermitian}). Case (2) is realised by the vertex-maximal clique incidence graph of the Johnson graph (see Section \ref{s:johnson}) while Case (3) is realised by the vertex-maximal clique incidence graph of the Hamming graph (Section \ref{s:hamming}). We do not have any example for case (4) that is not also an example for one of the previous cases.

We also characterise graphs of small minimum valency which are star-transitive and/or st(edge)-transitive.  In particular, we show the following.

\begin{theorem}
\label{thm:small valency}
Let $L$ be a connected graph with minimal valency one or two. Then $L$ is star-transitive and st(edge)-transitive if and only if one of the following holds:
\begin{itemize}
\item[(1)] $L$ is a complete bipartite graph $K_{1,n}$ for some $n \ge 1$;
\item[(2)] $L$ is a cycle of length $n$ for some $n \ge 3$; or
\item[(3)] there exists a locally fully symmetric, arc-transitive graph $\Sigma$ of
valency at least three such that $L$ can be obtained by subdividing each edge 
of $\Sigma$ with a new vertex of valency two.
\end{itemize}
\end{theorem}

 In Section~\ref{s:small}, we characterise star-transitive and st(edge)-transitive graphs of small girth and minimal valency and prove Theorem~\ref{thm:small valency}. In Section~\ref{s:observations}, we prove some preliminary general results about star-transitive and st(edge)-transitive graphs. The proofs of Theorems~\ref{v-star-st(edge)-t} and \ref{vtx-intrans} are given in Sections~\ref{s:vertex trans} and \ref{s:vertex intran}, respectively. Finally, in Section~\ref{s:examples}, we describe many examples of graphs that are both star-transitive and st(edge)-transitive.

We expect that our results and examples will contribute to the program described in \cite{FHT} of investigating the automorphism groups of nonpositively curved polygonal complexes and their lattices.  In particular, if $X$ is a simply-connected $(k,L)$-complex with $L$ star-transitive and st(edge)-transitive, then basic questions for which our results may prove useful include whether $\Aut(X)$ is discrete, and whether $\Aut(X)$ admits a lattice.

\subsection*{Acknowledgements}

We thank Nir Lazarovich and Michah Sageev for providing us with the preprint \cite{L}.  We also thank Hendrik van Maldeghem for pointing us to the references for local actions of generalised polygons.

\section{Background on $(k,L)$-complexes}\label{s:algtop}

In Section \ref{sec:algtopggt} we briefly recall several key definitions from algebraic topology and geometric group theory, and apply these in the setting of $(k,L)$-complexes.  We then discuss in Section \ref{sec:incidence} the relationship between $(k,L)$-complexes and rank $3$ incidence geometries.

\subsection{Definitions from algebraic topology and geometric group theory}\label{sec:algtopggt}

We first recall the definition of a \emph{$2$-dimensional CW-complex}, also known as a \emph{$2$-dimensional cell complex}.  (A reference is, for example,  \cite{Hatcher}.)   Denote by $D^1$ the closed interval $[-1,1]$ with boundary $\partial D^1$ the points $S^0 = \{-1,1\}$, and by $D^2$ the closed unit disk in the Euclidean plane with boundary $\partial D^2$ the unit circle $S^1$.
A space $X$ is a $2$-dimensional CW-complex if it is constructed as follows:
\begin{enumerate}
\item Begin with a discrete set $X^{(0)}$, called the \emph{$0$-skeleton}, whose points are the \emph{$0$-cells}.
\item The \emph{$1$-skeleton} $X^{(1)}$ is the quotient space obtained from the disjoint union $X^{(0)} \sqcup_\alpha D^1_\alpha$, of $X^{(0)}$ with a collection of closed intervals $D^1_\alpha$, by identifying each boundary point $x \in \partial D^1_\alpha$ with a $0$-cell $\varphi_\alpha(x) \in X^{(0)}$.  That is, each $\varphi_\alpha$ is a function from $S^0 = \{-1,1\}$ to $X^{(0)}$ (necessarily continuous).  We equip $X^{(1)}$ with the quotient topology. The images of the $D^1_\alpha$ in $X^{(1)}$ are called the (closed) \emph{$1$-cells}.  
\item The \emph{$2$-skeleton} $X^{(2)}$ is the quotient space obtained from the disjoint union $X^{(1)} \sqcup_\beta D^2_\beta$, of $X^{(1)}$ with a collection of closed disks $D^2_\beta$, by identifying each boundary point $x \in \partial D^2_\beta$ with a point $\varphi_\beta(x) \in X^{(1)}$, where each $\varphi_\beta$ is a continuous function from the circle $\partial D^2_\beta = S^{1}$ to $X^{(1)}$.  We equip $X^{(2)}$ with the quotient topology. The images of the $D^2_\beta$ in $X^{(2)}$ are called the (closed) \emph{$2$-cells}.
\item Since $X$ is $2$-dimensional, $X$ is equal to its $2$-skeleton $X^{(2)}$. 
\end{enumerate}
The maps $\varphi_\alpha$ and $\varphi_\beta$ are called the \emph{attaching maps}.  The (closed) \emph{cells} of $X$ are its (closed) $0$-, $1$- and $2$-cells.  The $1$-skeleton of $X$ may be thought of as a graph (not necessarily simple), with vertex set the $0$-skeleton and edges the $1$-cells.  The additional conditions required in order for a $2$-dimensional CW-complex $X$ to be a \emph{polygonal complex} are stated in the introduction.

We next recall the definitions of geodesic metric spaces and the curvature conditions $\CAT(0)$ and $\CAT(-1)$.  For details on this material, see \cite{BH}.  Let $(X,d_X)$ be a metric space.  A continuous function $\gamma:[a,b] \to X$ (for $a < b$ real numbers) is a \emph{geodesic} if for all $a \leq t < t' \leq b$, we have $d_X(\gamma(t),\gamma(t')) = t' - t$.   The metric space $(X,d_X)$ is \emph{geodesic} if for all $x,y \in X$, there is a geodesic $\gamma:[a,b]\to X$ such that $\gamma(a) = x$ and $\gamma(b) = y$.  We may denote this geodesic by $[x,y]$.  Note that there may be more than one geodesic connecting $x$ and $y$.  For example, in Euclidean space, each geodesic is a straight line segment and there is a unique geodesic connecting each pair of points, while on the sphere $S^2$ with its usual metric, each geodesic is an arc of a great circle, and antipodal points are connected by infinitely many geodesics.

Let $(X,d_X)$ be a geodesic metric space.  A \emph{geodesic triangle} in $X$ is a triple of points $x,y,z$, together with a choice of geodesics $[x,y]$, $[y,z]$ and $[z,x]$.  Given a geodesic triangle $\Delta = \Delta(x,y,z)$, a \emph{comparison triangle} in the Euclidean plane is a triple of points $\bar{x}, \bar{y}, \bar{z}$ such that $d_X(x,y) = d(\bar{x},\bar{y})$, $d_X(y,z) = d(\bar{y}, \bar{z})$ and $d_X(z,x) = d(\bar{z} , \bar{x})$, where $d$ is the Euclidean metric.  For each point $p \in [x,y]$, there is a \emph{comparison point} denoted $\bar{p}$ in the straight line segment $[\bar{x}, \bar{y}]$, with the comparison point $\bar{p}$ defined by the equation $d_X(x,p) = d(\bar{x} , \bar{p})$.  Similarly, we define comparison points for $p$ in $[y,z]$ and $[z,x]$.  The space $X$ is said to be $\CAT(0)$ if for every geodesic triangle $\Delta = \Delta(x,y,z)$, and every pair of points $p,q \in [x,y] \cup [y,z] \cup [z,x]$, we have $d_X(p,q) \leq d(\bar{p} ,\bar{q})$.  Roughly speaking, triangles in a $\CAT(0)$ space are ``no fatter" than Euclidean triangles.  A $\CAT(0)$ space is sometimes said to be \emph{nonpositively curved}.  

We may instead consider comparison triangles in the hyperbolic plane, and define $X$ to be $\CAT(-1)$ or \emph{negatively curved} if its triangles are ``no fatter" than hyperbolic triangles.  Every $\CAT(-1)$ space is also a $\CAT(0)$ space (Theorem II.1.12 of \cite{BH}).  Key properties of $\CAT(0)$ spaces $X$ include that each pair of points in $X$ is connected by a unique geodesic, that $X$ is contractible hence simply-connected and that if a group acting by isometries on $X$ has a bounded orbit, then it fixes a point (see, respectively, Proposition II.1.4, Corollary II.1.5 and Corollary II.2.8  of \cite{BH}).

We now consider the case of interest, when $X$ is a $(k,L)$-complex, as defined in the introduction.  Assume that each face of $X$ is metrised as a regular Euclidean $k$-gon, or that each face of $X$ is metrised as a regular hyperbolic $k$-gon.  Then by Theorem I.7.50 of \cite{BH}, $X$ is a complete geodesic metric space when equipped with the ``taut string" metric, in which each geodesic is a concatenation of a finite number of geodesics contained in faces. 

If the faces of a $(k,L)$-complex $X$ are regular Euclidean $k$-gons and $(k,L)$ satisfies the Gromov Link Condition, then $X$ is locally $\CAT(0)$ (see I.5.24 of \cite{BH}).  Hence by the Cartan--Hadamard Theorem \cite[II.4.1]{BH}, the universal cover of $X$ is a $\CAT(0)$ space.  Similarly, if either $k > m$ and $\girth(L) \geq n$, or $k \geq m$ and $\girth(L) > n$,  for $(m,n) \in \{ (3,6), (4,4), (6,3) \}$, then the faces of $X$ may be metrised as regular hyperbolic $k$-gons with vertex angles $2\pi/\girth(L)$, the complex $X$ is locally $\CAT(-1)$ and the universal cover of $X$ is thus $\CAT(-1)$.  We will henceforth be considering simply-connected $(k,L)$-complexes $X$ which satisfy the Gromov Link Condition, and so  are $\CAT(0)$.

\subsection{Relationship with incidence geometries}\label{sec:incidence}
A $(k,L)$-complex may be viewed combinatorially as a rank three incidence structure, namely, a geometry consisting of three types of objects, vertices, edges and faces, such that each edge is incident with exactly two vertices, each face is incident with exactly $k$ edges and $k$ vertices, and the graph with vertex set the edges incident with a given vertex and edges the faces is  isomorphic to $L$. In the notation of Buekenhout \cite{Bu}, the geometry has diagram
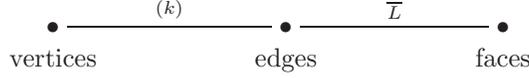
\[\xymatrix@R=0mm{
{\bullet}   \ar@{-}[rr]^{(k)} &   & {\bullet}   \ar@{-}[rr]^{\overline{L}} &  & {\bullet} 
\\
{\mbox{vertices}} & & {\mbox{edges}} & & {\mbox{faces}} 
}\]
where the label $(k)$ denotes the vertex-edge incidence graph of a  $k$-cycle and $\overline{L}$ denotes the vertex-edge incidence graph of $L$. 

We call a $(k,L)$-complex \emph{connected} if both the link graph and the vertex-edge incidence graph are connected. A \emph{flag} in the geometry is an incident vertex-edge-face triple. All  connected $(k,L)$-complexes that can be embedded in $\mathbb{E}^3$ and have a group of automorphisms acting regularly on flags were classified by Pellicer and Shulte \cite{PS1,PS2}. The only finite ones were seen to be the eighteen finite regular polyhedra. Polyhedra are precisely the finite  $(k,L)$-complexes with $L$ a cycle.



\section{Graph- and group-theoretic definitions and notation}
\label{s:definitions}

In Section \ref{sec:graphs} we discuss the main definitions from graph theory that we will require, and then in Section~\ref{sec:groups} we define the group-theoretic notation that we will use.  For ease of comparison with the literature, we now switch to standard graph-theoretic notation.  We also assume some basic definitions from algebraic graph theory \cite{GR} and from the theory of permutation groups \cite{DM}. All graphs considered in this paper are finite. 

\subsection{Graph theory}\label{sec:graphs}

Let $\G$ be a graph with vertex set $\VG$ and edge set $\EG$. If $\G$ is {\em simple}, that is, a graph without loops or multiple edges, and $e \in \EG$ connects the vertices $u$ and $v$, then we identify the (undirected) edge $e$ with the set $\{u,v\}$, and we denote the {\em arc} (directed edge) from $u$ to $v$ by $uv$.  For each vertex $v$ we denote by $\G(v)$ the set of neighbours of $v$, that is, the set of vertices adjacent to $v$. The set of all vertices at distance $i$ from $v$ will be denoted by $\Gamma_i(v)$. In particular, $\Gamma_1(v)=\Gamma(v)$. For $X \subseteq \VG$, the {\em restriction of $\G$ to $X$} is the graph $\G\mid_X$ with vertex set $X$ and edge set consisting of those edges of $\G$ that have both endpoints in $X$. The {\em girth} $\girth(\G)$ is the length of the shortest cycle in $\G$.  The \emph{valency} of a vertex is the number of neighbours that it has and the graph $\G$ is called {\em $k$-regular} if each $v \in \VG$ has valency $k$. A $3$-regular graph is also called a \emph{cubic graph}. If $\G$ is $k$-regular then we also say that $\G$ {\em has valency $k$}. A graph is called \emph{regular} if it is $k$-regular for some $k$. If $\G$ is bipartite and vertices in the two parts of the bipartition have valency $\ell$ and $r$, respectively, then we say that $\G$ {\em has bi-valency $\{ \ell,r \}$}.

If $G$ is a group of automorphisms of $\G$ and $v \in \VG$ then $G_v$ denotes the stabiliser in $G$ of the vertex $v$. If $X \subseteq \VG$ is stabilised setwise by a subgroup $H \le G$ then we denote by $H^X$ the permutation group induced by $H$ on $X$. In particular, $G_v^{\G(v)}$ is the group induced on $\G(v)$ by $G_v$. We say that $G$ is {\em locally transitive}, {\em locally primitive}, {\em locally $2$-transitive}, or {\em locally fully symmetric} if for each $v \in \VG$ the group $G_v^{\G(v)}$ is transitive, primitive, $2$-transitive, or the symmetric group on $\G(v)$, respectively. The graph $\G$ is locally transitive, locally primitive, locally $2$-transitive, or locally fully symmetric if there exists $G \le \Aut(\G)$ with the appropriate property (equivalently, as all four properties hold in overgroups, $\Aut(\G)$ has the appropriate property). 

For an edge $\{u,v\}$ we define $G_{\{u,v\}}$ to be the setwise stabiliser of $\{u,v\}$, and for an arc $uv$, we define $G_{uv}:= G_u \cap G_v$.  Let $d$ be the usual distance function on $\G$, so that each edge has length $1$.  Then for each natural number $n$ and each $v \in \VG $, we define 
\[ \Gvn := \{ g \in G_v \mid w^g = w \, \forall w \in \VG \mbox{ such that } d(v,w) \leq n \}  \]
as the pointwise stabiliser of the ball of radius $n$ around $v$. For $\{ u,v \} \in \EG$, $G_{uv}^{[1]}:= G_u^{[1]} \cap G_v^{[1]}$. 

An \emph{$s$-arc} in a graph $\Gamma$ is an $(s+1)$-tuple $(v_0,v_1,\ldots,v_s)$ of vertices such that $\{ v_i,v_{i+1}\} \in \EG$ and $v_{i-1}\neq v_{i+1}$, that is, it is a walk of length $s$ that does not immediately turn back on itself. Let $G \le \Aut(\Gamma)$. We say that $\Gamma$ is \emph{locally $(G,s)$-arc transitive} if for each vertex $v$, the stabiliser $G_v$ acts transitively on the set of $s$-arcs of $\Gamma$ starting at $v$. If $G$ is transitive on the set of all $s$-arcs in $\Gamma$ then we say that $\Gamma$ is \emph{$(G,s)$-arc transitive}. If all vertices of $\Gamma$ have valency at least two then locally $s$-arc transitive implies locally $(s-1)$-arc transitive. Moreover, $s$-arc transitive implies locally $s$-arc transitive. Conversely, if $G$ is transitive on $\VG$ and $\Gamma$ is locally $(G,s)$-arc transitive then $\Gamma$ is $(G,s)$-arc transitive. We observe that a graph with all vertices having valency at least 2 is locally $(G,2)$-arc transitive if and only if $G_v^{\Gamma(v)}$ is 2-transitive for all vertices $v$ (see for example \cite[Lemma 3.2]{GLP1}). Moreover, if $\G$ is locally $G$-transitive then $G$ acts transitively on $\EG$ and either $G$ is transitive on $\VG$ or $\G$ is bipartite and $G$ acts transitively on both sets of the bipartition. If $\G$ is $(G,s)$-arc-transitive but not $(G,s+1)$-arc-transitive then we say that $\G$ is {\em $(G,s)$-transitive}. Finally, if $G=\Aut(\G)$ then we drop the name $G$ from all notation introduced in this paragraph and say that $\G$ is \emph{locally $s$-arc-transitive}, \emph{$s$-arc-transitive}, and \emph{$s$-transitive}, respectively. 

The study of $s$-arc transitive graphs goes back to the  seminal work of Tutte \cite{tutte47,tutte59} who showed that a cubic graph  is at most 5-arc transitive. This was later extended by Weiss \cite{W2} to show that any graph of valency at least three is at most 7-arc transitive. Weiss \cite{Weiss78} also showed that a cubic graph is at most locally 7-arc transitive while Stellmacher \cite{sleq9} has announced that a graph of valency at least 3 is at most locally 9-arc transitive. In each case the upper bound is met. Note that a cycle is $s$-arc transitive for all values of $s$.

The following definitions, using a slightly different language, were introduced by Lazarovich \cite{L}.
\begin{definition}  
\label{def:star and iso}
{\em Let $\G$ be a simple graph, let $v \in \VG$, and let $e=\{ u,v\} \in \EG$. 
The \emph{open star of $v$}, denoted $\st(v)$, is the union of $\{v \}$ and the set $\{ f \in \EG \mid f \mbox{ is incident to } v \}$.  Similarly, the \emph{open edge-star of $e$}, denoted $\st(e)$ or $\st(\{u,v\})$, is the union of the sets  $\{ u \}$, $\{ v \}$, and  $\{ f \in \EG \mid f \mbox{ is incident to at least one of } u,v \}$.

Given two open stars $\st(v_1)$ and $\st(v_2)$, a {\em star isomorphism} is a bijection $\varphi: \st(v_1) \to \st(v_2)$ such that $\varphi(v_1)=v_2$. 

Given two open edge-stars $\st(\{u_1,v_1\})$ and $\st(\{u_2,v_2\})$, an {\em edge-star isomorphism}  is a bijection $\varphi: \st(\{u_1,v_1\}) \to \st(\{u_2,v_2\})$ such that
\begin{itemize}
\item[(i)] $\varphi(\VG \cap \st(\{u_1,v_1\}))=\VG \cap \st(\{u_2,v_2\})$, that is, the vertices $u_1,v_1$ are mapped (in some order) to the vertices $u_2,v_2$.
\item[(ii)] $\varphi$ is incidence-preserving, that is, $f \in \EG \cap \st(\{u_1,v_1\})$ is incident to $u_1$ if and only if $\varphi(f)$ is incident to $\varphi(u_1)$ and 
$f \in \EG \cap \st(\{u_1,v_1\})$ is incident to $v_1$ if and only if $\varphi(f)$ is incident to $\varphi(v_1)$. In particular, $\varphi(\{ u_1,v_1 \}) = \{ u_2,v_2 \}$. 
\end{itemize}}
\end{definition}

\begin{definition}
\label{def:star trans}
{\em Let $\G$ be a simple graph, and let $G \le \Aut(\G)$. Then $\G$ is called 
\begin{itemize}
\item[(i)]  \emph{$G$-star-transitive} if for all $v_1,v_2\in \VG$ and for all star isomorphisms $\varphi: \st(v_1) \to \st(v_2)$, there exists an automorphism $\psi \in G$
such that $\psi(v_1)=\varphi(v_1)$ and for all $f \in \EG \cap \st(v_1)$ we have $\psi(f)=\varphi(f)$; and 
\item[(ii)]  \emph{$G$-st(edge)-transitive} if for all $\{u_1,v_1\}, \{u_2,v_2\} \in \EG$ and edge-star isomorphisms  $\varphi: \st(\{u_1,v_1\}) \to \st(\{u_2,v_2\})$,
there exists an automorphism $\psi \in G$
such that $\psi(u_1)=\varphi(u_1)$,  $\psi(v_1)=\varphi(v_1)$, and for all $f \in \EG \cap \st(\{u_1,v_1\})$ we have $\psi(f)=\varphi(f)$.
\end{itemize}}
If $G =\Aut(\G)$ then we simply say that $\G$ is star-transitive or st(edge)-transitive, respectively. 
\end{definition}

A subtlety of Definition \ref{def:star trans} is that if there is no star-isomorphism $\st(v_1) \to \st(v_2)$ or edge-star isomorphism $\st(\{u_1,v_1\}) \to \st(\{u_2,v_2\})$, then the required property of extending to a graph automorphism holds trivially. 
Another subtlety is the introduction of the notions of star-transitivity and st(edge)-transitivity relative to subgroups $G \le \Aut(\G)$. Considering subgroups of $\Aut(\G)$ with certain transitivity properties is quite common in algebraic graph theory; the main reason is that there are examples where some $G \le \Aut(\G)$ extends to covers of $\G$ but the full automorphism group does not. For example, the icosahedron is a cover of the complete graph $K_6$ but not all of $\Aut(K_6)=S_6$ extends.

The reason for the somewhat cumbersome formulation of the definitions above is that in the case when $\girth(\G) \le 4$, the definition of a star isomorphism or edge-star isomorphism $\varphi$ does {\em not} 
require that $\varphi$ preserves the possible adjacency relations among the neighbours of the vertices occurring in the open stars and edge-stars. However, the graph automorphisms defined in Definition~\ref{def:star trans}, extending the star isomorphisms and edge-star isomorphisms, must preserve such adjacency relations.  The following result, whose proof is immediate, says that for large enough girth we can work with much simpler definitions.  Given a vertex $v$ we let  $X(v):=\{ v \} \cup \G(v)$, and for an edge $\{u,v\}$ we let
$X(\{u,v\}):=\{ u \} \cup \{ v \} \cup \G(u) \cup \G(v)$.

\begin{prop}
\label{prop:equiv}
$(i)$ If $\girth(\G) \ge 4$ then for $v \in \VG$, $\st(v)$ can be identified with the restriction 
$\G \mid_{X(v)}$. A star isomorphism is then 
a graph isomorphism $\varphi_1: \G \mid_{X(v_1)} \to \G \mid_{X(v_2)} $ and 
$\G$ is star-transitive if and only if 
every  star isomorphism extends to an automorphism of $\G$. 

$(ii)$ If $\girth(\G) \ge 5$ then for $\{u,v\} \in \EG$, $\st(\{u,v\})$ can be identified with the restriction $\G \mid_{X(\{u,v\})}$. An edge-star isomorphism is then a graph isomorphism 
$\varphi_2: \G \mid_{X(\{u_1,v_1\})} \to \G \mid_{X(\{u_2,v_2\})}$ and the graph $\G$ is st(edge)-transitive if and only if 
every edge-star isomorphism extends to an automorphism of $\G$.
\end{prop}

\subsection{Group-theoretic notation}\label{sec:groups}

For a natural number $k$, we denote by $S_k$ the symmetric group on $k$ letters, by $A_k$ the alternating group on $k$ letters, by $C_k$ the cyclic group of order $k$, and by $D_{2k}$ the dihedral group of order $2k$. The projective special linear group, projective general linear group, and projective semilinear group of dimension $d$ over a field of size $q$ is denoted by $\PSL(d,q)$, $\PGL(d,q)$, and $\PGaL(d,q)$, respectively. Given a group $A$ and a natural number $k$, we denote by $A \Wr S_k$ the following wreath product: let $B$ be the direct product of $k$ copies of $A$.  Then $S_k$ acts naturally on $B$ by permuting the $k$ copies of $A$, and $A \Wr S_k$ is the semidirect product induced by this action. We denote the semidirect product of two groups $A$ and $B$ by $A{:}B$. A group that is a (not necessarily split) extension of a subgroup $A$ by a group $B$ will be denoted by $A.B$. Given a prime $p$, $p^m$ will be used to denote an elementary abelian group of order $p^m$, and we will use $[p^m]$ to denote a group of order $p^m$ when we do not wish to specify the isomorphism type.  The \emph{socle} of a group $G$ is the subgroup generated by all the minimal normal subgroups of $G$, and is denoted by $\soc(G)$. When $n\geq 5$ or $n=3$, we have that $\soc(S_n)=A_n$.

We refer to a triple of groups $(A,B,A\cap B)$ as an \emph{amalgam}. A \emph{completion} of the amalgam $(A,B,A\cap B)$ is a group $G$ together with group homomorphisms $\phi_1:A\rightarrow G$ and $\phi_2:B\rightarrow G$ such that $\phi_1$ and $\phi_2$ are one-to-one, $G=\langle \phi_1(A),\phi_2(B)\rangle$ and $\phi_1(A)\cap\phi_2(B)=\phi_1(A\cap B)=\phi_2(A\cap B)$.

\section{Graphs with small girth or with small minimal valency}
\label{s:small}

In this section, we characterise star-transitive and st(edge)-transitive graphs of girth at most four, so that in the rest of the paper we can concentrate on the case $\girth(\G) \ge 5$ and use the simplified description of stars and edge-stars, as given in Proposition~\ref{prop:equiv}. We also characterise star-transitive and st(edge)-transitive graphs of minimal valency one or two. 

\begin{lemma}
\label{lem:girth 3}
The only connected star-transitive graphs of girth $3$ are the complete graphs $K_n$, for some $n \ge 3$. The only connected st(edge)-transitive graph of girth $3$ is the triangle $K_3$. The only connected st(edge)-transitive graphs of girth $4$ are the complete bipartite graphs $K_{m,n}$, for some $m,n \ge 2$. 
\end{lemma}

\begin{proof}
Let $\{ u,v,w \}$ be a cycle of length $3$ in $\G$, and suppose that $\G$ is star-transitive. Considering the extensions of all star isomorphisms $\varphi: \st(v) \to \st(v)$ to 
automorphisms of $\G$, we obtain that any two vertices in $\G(v)$ are adjacent. Hence, for any $x \in \G(v)$, $x$ is contained in a cycle of length $3$, and by the same argument as above
all neighbours of $x$ are adjacent. In particular, all $y \in \G(x) \setminus \{ v\}$ are adjacent to $v$, and so $\{ v \} \cup \G(v) = \{ x\} \cup \G(x)$. 
As $\G$ is connected, we obtain that $\VG = \{ v \} \cup \G(v)$ and $\G$ is a complete graph.

Suppose now that  $\{ u,v,w \}$ is a cycle of length $3$ in $\G$, and that $\G$ is st(edge)-transitive. If the vertex $u$ has valency greater than $2$ then there exists an edge-star isomorphism 
$\varphi: \st(\{u,v\}) \to \st(\{u,v\}) $ such that $\varphi(u)=u$, $\varphi(v)=v$, $\varphi(\{u,w\}) \ne \{ u,w\}$, and  $\varphi(\{v,w\}) = \{ v,w\}$. However, $\varphi$ cannot be
extended to an automorphism of $\G$, a contradiction. Similarly, $v$ and $w$ also must be of valency $2$ and so, as $\G$ is connected, $\VG = \{ u,v,w \}$.

Finally, suppose that $\girth(\G)=4$, $\G$ is st(edge)-transitive, and let $\{ u,v,w,z \}$ be a $4$-cycle. Considering the extensions of all edge-star isomorphisms $\varphi: \st(\{u,v\}) \to \st(\{ u,v\})$ that fix $u$ and $v$, we obtain
that, as an image of the edge $\{ w,z\}$, every pair $\{ w_1,z_1 \}$ with 
$w_1 \in \G(v)$ and $z_1 \in \G(u)$ is in $\EG$. Therefore, for every 
$w_1 \in \G(v)$ we have $\G(w_1) \supseteq \G(u)$. Repeating the same argument
with edge-star isomorphisms $\varphi: \st(\{w_1,v\}) \to \st(\{ w_1,v\})$ and 
a four-cycle containing $\{ w_1,z \}$, we obtain that for every 
$w_2 \in \G(v)$ we have $\G(w_2) \supseteq \G(w_1)$. In particular, for $w_2=u$,
$\G(w_1)= \G(u)$. As $w_1$ was an arbitrary element of $\G(v)$, all vertices in $\G(v)$ have the same neighbours. Similarly, all vertices in $\G(u)$ have the same neighbours, $\VG= \G(u) \cup \G(v)$, and $\G$ is a complete bipartite graph. 
\end{proof}

Next, we characterise the star-transitive and st(edge)-transitive graphs with a vertex of valency one. For $n \ge 3$, we define the {\em spider graph $T_n$} as a graph with $2n+1$ vertices  
$V(T_n)=\{ x,y_1,\ldots,y_n,z_1,\ldots,z_n \}$ and $2n$ edges 
$E(T_n)= \{ \{ x,y_i \}, \{ y_i,z_i \} \mid 1 \le i \le n \}$. 

\begin{lemma}
\label{lem:valency one}
The only connected star-transitive graphs with a vertex of valency one are the
complete bipartite graphs $K_{1,n}$, for some $n \ge 1$. 
The only connected st(edge)-transitive graphs with a vertex of valency one are: the complete bipartite graphs $K_{1,n}$, for some $n \ge 1$; the path $P_4$ with four vertices; and the spider graphs $T_n$, for some $n \ge 3$. 
\end{lemma}

\begin{proof}
Let $\G$ be a simple graph, let $v \in \VG$ have valency one, and let $u$ be the unique neighbour of $v$. If $\G$ is star-transitive then, considering the star-isomorphisms $\varphi: \st(u) \to \st(u)$ mapping $v$ to other neighbours of $u$, we obtain that all neighbours of $u$ have valency one. As $\G$ is connected, we obtain $\G \cong K_{1,n}$, where $n$ is the valency of $u$. 

Suppose now that $\G$ is st(edge)-transitive. We distinguish three cases, according to the valency of $u$. If $u$ has valency at least three then let $w,z$ be neighbours of $u$ that are different from $v$. Considering edge-star isomorphisms $\varphi: \st(\{u,w\}) \to \st(\{u,w\})$ that fix $u$ and map $v$ to neighbours of $u$ different from $w$, we obtain that all neighbours of $u$ different from $w$ have valency one. Repeating the same process with edge-star isomorphisms 
$\varphi: \st(\{u,z\}) \to \st(\{u,z\})$, we deduce that $w$ also has valency one and $\G \cong K_{1,n}$, where $n$ is the valency of $u$.

If $u$ has valency one then $\G \cong K_{1,1}$. If $u$ has valency two then let $x$ be the neighbour of $u$ different from $v$. We distinguish three subcases, according to the valency of $x$. If $x$ has valency one then $\G \cong K_{1,2}$. If $x$ has valency two then, from the edge-star isomorphism 
$\varphi: \st(\{u,x\}) \to \st(\{u,x\})$ that exchanges $u$ and $x$, we obtain that $\G \cong P_4$. Finally, if the valency of $x$ is at least three then let 
$w,z$ be neighbours of $x$ that are different from $u$. Considering edge-star isomorphisms $\varphi: \st(\{x,w\}) \to \st(\{x,w\})$ that fix $x$ and map $u$ to neighbours of $x$ different from $w$, we obtain that all neighbours of $x$ different from $w$ have valency two and they are adjacent to a vertex of valency one. 
Repeating the argument with edge-star isomorphisms $\varphi: \st(\{x,z\}) \to \st(\{x,z\})$, we see that the neighbour $w$ also has this property and so $\G \cong T_n$, where $n$ is the valency of $x$. 
\end{proof}

Finally, we handle the case of minimal valency two. For any $n \ge 3$, the cycle $C_n$ is $2$-regular, star-transitive, and st(edge)-transitive. We obtain further examples by the following constructions. 

Let $\Sigma$ be a simple graph of minimal valency at least three. We construct
the {\em $1$-subdivision} $\G$ of $\Sigma$ by replacing each edge by a path of length two. Formally, we define $\VG=V(\Sigma) \cup E(\Sigma)$. The sets 
$V(\Sigma)$ and $ E(\Sigma)$ are independent in $\G$, and $v \in V(\Sigma)$
is connected to $e \in E(\Sigma)$ in $\G$ if and only if $v$ and $e$ are
incident in $\Sigma$.
Similarly, we construct the {\em $2$-subdivision} of $\Sigma$ by replacing each edge by a path of length three. The following proposition is easy to verify.

\begin{prop}
\label{min two examples}
Let  $\Sigma$ be an arc-transitive graph of minimal valency at least three which is locally fully symmetric. Then the $1$-subdivision of $\Sigma$ is both star-transitive and st(edge)-transitive. The $2$-subdivision of $\Sigma$ is st(edge)-transitive, but not star-transitive. 
%
\end{prop}

\begin{lemma}
\label{min two star}
Suppose that $\G$ is star-transitive and the minimal valency in $\G$ is two, but $\G$ is not $2$-regular. Then there exists an arc-transitive graph $\Sigma$ of valency at least three which is locally fully symmetric such that $\G$ is isomorphic to the $1$-subdivision of $\Sigma$.
\end{lemma}

\begin{proof}
Since $\Gamma$ is not 2-regular, there exists $\{ v,w \} \in \EG$ with $v$ of valency $k>2$ and $w$ with valency $2$. Lazarovich~\cite[Lemma 1.1]{L} proved that this implies $\G$ is edge-transitive; consequently, all edges of $\G$ connect valency $2$ vertices with vertices of valency $k$. Hence $\G$ is bipartite and $\G$ is a $1$-subdivision of a graph $\Sigma$ with minimal valency at least $3$.

Automorphisms of $\G$, restricted to the vertices of valency $k$, naturally define automorphisms of $\Sigma$. Star isomorphisms $\st(v) \to \st(v)$, with $v \in \VG$ and $v$ of valency $k$, show that $\Sigma$ is locally fully symmetric, and consequently $\Sigma$ is edge-transitive. Finally, star isomorphisms $\st(w) \to \st(w)$ of $\G$, with $w \in \VG$ of valency two, show that edges in $\Sigma$ can be turned around by automorphisms, and so $\Sigma$ is arc-transitive.  
\end{proof}

\begin{lemma}
\label{min two edge}
Suppose that $\G$ is st(edge)-transitive, the minimal valency in $\G$ is two, but $\G$ is not $2$-regular. Then there exists an edge-transitive graph $\Sigma$ of minimal valency at least three which is locally fully symmetric such that one of the following holds. 
\begin{itemize}
\item[$(1)$] $\Sigma$ is non-regular and $\G$ is isomorphic to the $1$-subdivision of $\Sigma$.
\item[$(2)$] $\Sigma$ is arc-transitive, and $\G$ is isomorphic to the $1$- or $2$-subdivision of $\Sigma$.
\end{itemize}
\end{lemma}

\begin{proof}
We claim that there are no three vertices $u,v,w$, all of valency two, such that $\{ u,v\} \in \EG$ and $\{ v,w\} \in \EG$. Indeed, if $u$, $v$, $w$ are such vertices then there is a unique path in $\G$ starting with the edge $\{ u,v\}$, consisting of vertices of valency two, such that the endpoint $x$ of the path has a neighbour $z$ of valency greater than two. Then, for the last two vertices $x$ and $y$ of this path, the edge-star isomorphism $\varphi: \st(\{ x,y \}) \to \st(\{ x,y \})$ that exchanges $x$ and $y$ has no extension to an automorphism of $\G$, a contradiction.

Let $\{ v,w \} \in \EG$ with $v$ of valency greater than $2$ and $w$ with valency $2$, let $x,y$ be two further neighbours of $v$, and let $m$ be the maximal number of vertices on a path starting at $w$ and consisting of vertices of valency $2$. By the claim in the previous paragraph, $m \in \{ 1,2\}$. Considering the edge-star isomorphisms 
\begin{equation}
\label{eq:xyv}
\st(\{ x,v \}) \to \st(\{ x,v \}) \mbox{ and } 
\st(\{ y,v \}) \to \st(\{ y,v \})
\end{equation}
that fix the vertex $v$, we obtain that
all neighbours of $v$ have valency $2$ and for each neighbour $z$, the maximal length of a path starting at $z$ and consisting of vertices of valency $2$ is $m$. Then, by induction on the distance from $v$, we get that all vertices $v'$ of valency greater than $2$ have this property, and so $\G$ is the $m$-subdivision of a graph $\Sigma$ of minimal valency at least $3$.

Automorphisms of $\G$, restricted to the vertices of valency greater than two, naturally define automorphisms of $\Sigma$. 
The edge-star isomorphisms in \eqref{eq:xyv} show that $\Sigma$ is locally fully symmetric, and consequently $\Sigma$ is edge-transitive. If $m=2$ and $vwab$ is a path in $\G$ connecting the vertices $v,b$ of valency greater than $2$ then the 
edge-star isomorphism  $\st(\{ w,a \}) \to \st(\{ w,a \})$ that exchanges $w$ and $a$ shows that $\Sigma$ is arc-transitive and we are in case $(2)$ of the lemma. If $m=1$ and $\Sigma$ is regular then let $vwb$ a path in $\G$ connecting 
the vertices $v,b$ of valency greater than $2$. The edge-star isomorphism  $\st(\{ w,v \}) \to \st(\{ w,b\})$ shows that $\Sigma$ is arc-transitive, and again we are in case $(2)$. Finally, if $\Sigma$ is non-regular then we are in case $(1)$.
\end{proof}

Combining the results of this section, we obtain Theorem~\ref{thm:small valency}.

\section{Connections among star-transitivity, st(edge)-transitivity, and arc-transitivity}\label{s:observations}

This section contains preliminary results used for the proofs of Theorem \ref{v-star-st(edge)-t} and \ref{vtx-intrans}.
We begin by recording the following result of Lazarovich \cite[Lemma 1.1]{L}:

\begin{lemma}\label{l:lazarovich} If $\G$ is a connected star-transitive graph then either:
\begin{enumerate}
\item\label{c:vt} $\G$ is 2-arc-transitive; or
\item\label{c:et} $\G$ is edge-transitive and bipartite, with $\VG = A_1 \sqcup A_2$, and there exist $d_1, d_2 \in \N$ so that for all $v \in A_i$, the vertex $v$ has valency $d_i$ ($i = 1,2$).
\end{enumerate}
\end{lemma}

\noindent It is noted in the proof of \cite[Lemma 1.1]{L} that in Case \eqref{c:et}, $d_1 \neq d_2$.  We will discuss both cases of Lemma \ref{l:lazarovich} further below.
Our first observations are as follows.   

\begin{lemma}\label{l:vertex transitive}  Let $\G$ be a $G$-star-transitive graph.  If $\G$ is $k$-regular then $\G$ is $G$-vertex-transitive.
\end{lemma}

\begin{lemma}\label{l:locally fully symmetric}  Let $\G$ be a $G$-star-transitive graph.   Then $G_v^{\G(v)} = S_{|\G(v)|}$ for all $v \in \VG$, that is, $\G$ is locally fully symmetric.
\end{lemma}

The converse of Lemma \ref{l:locally fully symmetric} does not hold, as there are graphs which are locally fully symmetric but are not star-transitive.  In fact, there are regular graphs which are locally fully symmetric but are not star-transitive. The following example was first described by Lipschutz and Xu \cite{LX}.  Let $G = \PGL(2,p)$ for $p$ prime, $p \equiv \pm1 \pmod{24}$.  Then $G$ is generated by subgroups $H \cong D_{24}$ and $K \cong S_4$, such that $H \cap K \cong D_8$.  The graph $\G$ is defined to be the bipartite graph with vertex set $G/H \sqcup G/K$ and edge set $G/(H \cap K)$, so that the edge $g(H \cap K)$, for $g \in G$, connects the vertices $gH$ and $gK$.  Then $\G$ is cubic and locally fully symmetric, since the natural left-action of $G$ induces $S_3$ at each vertex.  However, $\G$ is not vertex-transitive.

The following sufficient conditions for star-transitivity are easily verified.

\begin{lemma}\label{l:suff star}  If there exists a subgroup $G \le \Aut(\G)$ such that either:
\begin{enumerate}
\item\label{suff1} $G$ is locally fully symmetric and vertex-transitive; or 
\item\label{suff2} $G$ is locally fully symmetric and edge-transitive, and there are 
natural numbers $k \neq \ell$ such that each vertex has valency either $k$ or $\ell$;
\end{enumerate}
then $\G$ is $G$-star-transitive.
\end{lemma}

We now consider st(edge)-transitivity.  
Lazarovich's main results concern graphs which are both star-transitive and st(edge)-transitive. We shall prove that, with the exception of the small-valency cases handled in Section~\ref{s:small}, st(edge)-transitivity implies star-transitivity.
Recall that for $v \in \VG$ and $\{ u,v \} \in \EG$, we defined 
$X(v) = \{ v \} \cup \G(v)$ and $X(\{ u,v\})=\{ u \} \cup \{ v\} \cup 
\G(u) \cup \G(v)$.

\begin{lemma}\label{lem:stedge}\label{lem:edgeaction} A connected graph $\G$ with $G = \Aut(\G)$, minimal valency at least three and girth at least four is st(edge)-transitive if and only if it is edge-transitive and either:
\begin{enumerate}
\item there is a $k \in \N$ so that for all edges $\{ u,v\}$, $G_{\{u,v\}}^{X(\{u,v\})} = S_{k-1} \Wr S_2$, in which case $\G$ is $k$-regular; or 
\item there are $k,\ell \in \N$ with $k \neq \ell$ so that for all edges $\{ u,v\}$, $G_{\{u,v\}}^{X(\{u,v\})} = S_{k-1} \times S_{\ell-1}$, in which case $\G$ is $(k,\ell)$-biregular.
\end{enumerate}
\end{lemma}


\begin{proof}
Observe first that since the minimum valency of $\Gamma$ is at least three then $\Gamma$ is not a tree. If $\Gamma$ has girth  four and is st(edge)-transitive then Lemma \ref{lem:girth 3} implies that $\Gamma$ is complete bipartite. Thus $\Gamma$ is edge-transitive and (1) holds if $\Gamma$ is regular while (2) holds in $\Gamma$ is biregular. Conversely, assume that  $\Gamma$ has girth four, is edge-transitive and either (1) or (2) hold.  Let $\{u,v,w,z\}$ be a 4-cycle. Then $z^{G_{u,v,w}}=\Gamma(u)\backslash\{v\}$ and so $\Gamma(u)=\Gamma(w)$. Similarly we see that $\Gamma(v)=\Gamma(z)$ and so $\Gamma$ is complete bipartite and hence st(edge)-transitive.

If $\girth(\G) \ge 5$ then clearly $\Aut(\G \mid_{X(\{u,v\})}) \cong S_{k-1} \Wr S_2$ or
$S_{k-1} \times S_{\ell-1}$ in the cases $k=\ell$ and $k \ne \ell$, respectively. Moreover, by Proposition~\ref{prop:equiv}(ii), every edge-star isomorphism is a graph isomorphism, and  $\Gamma$ is st(edge)-transiitive if and only if $\Gamma$ is edge-transitive and every $\varphi \in \Aut(\G \mid_{X(\{u,v\})})$ extends to an automorphism in
$G_{\{u,v\}}$.  Since  the restriction of any 
$\psi \in G_{\{u,v\}}$ to $X(\{u,v\})$ is in $\Aut(\G \mid_{X(\{u,v\})})$ the result follows.
\end{proof}

\begin{lemma}
\label{lem:stimpliesstar}
Let $\Gamma$ be a $G$-st(edge)-transitive graph of minimum valency at least three. Then $\Gamma$ is $G$-star-transitive.
\end{lemma}

\begin{proof}
If $\girth(\G) \le 4$ then $\G \cong K_{m,n}$ by Lemma~\ref{lem:girth 3} and the statement of this lemma holds. Suppose that $\girth(\G) \ge 5$, let 
$\{u,v\}$ be an edge of $\Gamma$, let $k$ be the valency of $v$ and let $G=\Aut(\Gamma)$.  By Lemma \ref{lem:edgeaction}, $G_v^{\Gamma(v)\backslash \{u\}}=S_{k-1}$. Let $w\in\Gamma(v)\backslash\{u\}$. Then again by Lemma \ref{lem:edgeaction}, $G_v^{\Gamma(v)\backslash \{w\}}=S_{k-1}$. As $k \ge 3$, it follows that $G_v^{\Gamma(v)}=S_k$ and in particular $G_x^{\Gamma(x)}=S_{|\Gamma(x)|}$ for each vertex $x$. Thus $\Gamma$ is locally fully symmetric, hence locally 2-arc transitive and edge-transitive. 

If $u$ has valency $\ell\neq k$ then Lemma \ref{l:suff star}(\ref{suff2}) implies that $\Gamma$ is star-transitive.
If $u$ also has valency $k$ then, since $\Gamma$ is st(edge)-transitive, Lemma \ref{lem:edgeaction}(1) implies that there is an element of $G$ interchanging $u$ and $v$. Hence $G$ is arc-transitive and in particular vertex-transitive. Thus by Lemma \ref{l:suff star}(\ref{suff1}), $\Gamma$ is star-transitive.
\end{proof}

We now consider actions on arcs.

\begin{lemma}\label{l:locally 2}  If $\G$ is $G$-star-transitive then $G$ is locally 2-transitive on $\Gamma$ and thus $\Gamma$ is locally $(G,2)$-arc transitive.  
\end{lemma}

\begin{proof}  Since $G_v^{\G(v)} = S_{|\G(v)|}$ which is $2$-transitive, the graph $\G$ is locally $2$-transitive.
\end{proof}

It follows that if a connected star-transitive graph $\G$ is vertex-transitive then $\G$ is $2$-arc transitive, as was given in Case \eqref{c:vt} of Lemma \ref{l:lazarovich} above.

\begin{lemma} 
\label{lem:3 arc}
If $\G$ has minimal valency at least three and $\G$ is $G$-st(edge)-transitive then $\G$ is locally $(G,3)$-arc transitive.
\end{lemma}

\begin{proof}
By Lemma \ref{lem:stimpliesstar}, $\Gamma$ is $G$-star-transitive and so by Lemma \ref{l:locally 2}, $\Gamma$ is locally $(G,2)$-arc transitive. Let $(u,v,w)$ be a 2-arc of $\Gamma$. Since $\Gamma$ is $G$-st(edge)-transitive, Lemma \ref{lem:edgeaction} applied to the edge $\{v,w\}$ implies that $G_{uvw}^{\Gamma(w)\backslash\{v\}}=S_{|\Gamma(w)|-1}$. Hence $G_{uvw}$ acts transitively on the set of $3$-arcs starting with $(u,v,w)$. Thus $\Gamma$ is locally $(G,3)$-arc transitive.
\end{proof}

By Lemmas~\ref{lem:3 arc}, \ref{lem:stimpliesstar} and \ref{l:vertex transitive}, we have the following corollary. 

\begin{corollary}\label{c:3-arc transitive} Suppose $\G$ is a connected $k$-regular graph.  If $\G$ is $G$-star-transitive and $G$-st(edge)-transitive, then $\G$ is $(G,3)$-arc transitive.
\end{corollary}

\section{The vertex-transitive case}
\label{s:vertex trans}

In this section we prove Theorem \ref{v-star-st(edge)-t}, that is, we conduct a local analysis of vertex-transitive and st(edge)-transitive graphs. 
We recall that a group is called \emph{$p$-local} for some prime $p$ if it contains a normal $p$-subgroup. Part (1) of the following fundamental theorem was proven by 
Gardiner~\cite[Corollary 2.3]{G} and was also established by Weiss in \cite{W1}.  Part (2) of the theorem is due to Weiss~\cite{W2}. With the hypothesis as in Theorem~\ref{t:arc kernel}(2),
Weiss~\cite{W2} proved additional results on the structure of the point stabiliser $G_u$, which we shall recall as needed in the proofs of Lemmas~\ref{val=4} and \ref{val>4}.

\begin{theorem}\label{t:arc kernel} Let $\G$ be a connected graph and let $G \le \Aut(\G)$ be vertex-transitive and locally primitive.  Then there exists a prime $p$ such that for all arcs $uv$:
\begin{enumerate}
\item $G_{uv}^{[1]}$ is a $p$-group; and 
\item if in addition $G_{uv}^{[1]} \ne 1$ then 
$G_{uv}^{\G(u)}$ is $p$-local.
\end{enumerate}
\end{theorem}

Let $\Gamma$ be a connected graph, and let $G\le\Aut(\Gamma)$ be vertex-transitive.
Recall that a $(G,s)$-arc-transitive graph is called {\em $(G,s)$-transitive} if it is not $(G,s+1)$-arc-transitive.
For small valencies, the explicit structure of a vertex stabiliser is known. For example, in the cubic case we have the following result due to Tutte \cite{tutte47}, and Djokovi\v{c} and Miller \cite{DM2}.

\begin{theorem}
Let $\Gamma$ be a  cubic $(G,s)$-transitive graph. Then one of the following hold:
\begin{enumerate}
 \item $s=1$ and $G_v=C_3$;
\item $s=2$ and $G_v=S_3$;
\item $s=3$ and $G_v=S_3\times C_2$;
\item $s=4$ and $G_v=S_4$;
\item $s=5$ and $G_v=S_4\times C_2$.
\end{enumerate}
\end{theorem}
\begin{corollary}\label{cor:cubic}
Let $\Gamma$ be a cubic $(G,s)$-transitive graph. Then $\Gamma$ is $G$-star-transitive if and only if $s\geq 2$, while $\Gamma$ is $G$-star-transitive and $G$-st(edge)-transitive if and only if $s\geq 3$.
\end{corollary}

In the 4-regular case, a complete determination of the vertex stabilisers for 4-regular 2-arc transitive graphs was given by Potocnik \cite{P}, building on earlier work of Weiss \cite{W4}.

\begin{lemma}\label{val=4}
Let $\Gamma$ be $4$-regular, let $v \in \VG$, and let $G \le \Aut(\G)$. Then $\Gamma$ is $G$-star-transitive if and only if
one of the following statements holds:
\begin{enumerate}
\item $\Gamma$ is $(G,2)$-transitive, and $G_v=S_4$;

\item $\Gamma$ is $(G,3)$-transitive, and $G_v=S_4\times S_3$ or $G_v=(A_4\times C_3).2$ with the element of order 2 inducing a nontrivial automorphism of both $C_3$ and $A_4$;

\item $\Gamma$ is $(G,4)$-transitive, and $G_v=3^2{:}\GL(2,3)$;

\item $\Gamma$ is $(G,7)$-transitive, and $G_v=[3^5]{:}\GL(2,3)$.
\end{enumerate}
\end{lemma}
\begin{proof}
Suppose first that $G$ satisfies one of the conditions $(1)$--$(4)$. Then $G$ is locally $2$-transitive, so 
$G_v^{\G(v)}\cong A_4$ or $S_4$. None of the listed point stabilisers have a quotient group isomorphic to 
$A_4$, so $G_v^{\G(v)}\cong S_4$. Consequently, by Lemma~\ref{l:suff star},  $\Gamma$ is $G$-star-transitive.

Conversely, suppose that $\Gamma$ is $G$-star-transitive. 
By Lemma~\ref{l:locally fully symmetric}, $G_v^{\Gamma(v)}=S_4\cong\PGL(2,3)$. If $G_v^{[1]}=1$, then $G_v\cong G_v^{\Gamma(v)}=S_4$ and so the stabiliser $G_{uvw}$ of the 2-arc $uvw$ is isomorphic to $S_2$.  Hence $G_{uvw}$ is not transitive on the set of three 3-arcs beginning with $uvw$ and so $\Gamma$ is $(G,2)$-transitive.

Suppose that $G_v^{[1]}\not=1$ and let $\{ v,w\} \in \EG$.
If $G_{vw}^{[1]}=1$, then 
\[1\not=G_v^{[1]}\cong G_v^{[1]}/G_{vw}^{[1]}\cong (G_v^{[1]})^{\Gamma(w)}\lhd G_{vw}^{\Gamma(w)}\cong S_3.\]
Thus, $G_v^{[1]}=C_3$ or $S_3$, and so for $u\in\Gamma(v)\backslash\{w\}$ we have that $G_{uvw}$ induces either $C_3$ or $S_3$ on the set of three 3-arcs beginning with $uvw$. Hence $\Gamma$ is $(G,3)$-arc-transitive.   Moreover, since $G_{vw}^{[1]}=1$, it follows from \cite{W2} that  $\Gamma$ is not $(G,4)$-arc-transitive.  Since $S_3$ has no outer automorphisms, if $G_v^{[1]}=S_3$ then we must have $G_v=S_3\times S_4$. If $G_v^{[1]}=C_3$ then $G_v=C_3 \times S_4$ or $(C_3\times A_4).2$ with the element of order 2 inducing a nontrivial automorphism of both $C_3$ and $A_4$. However, the first case does not occur (see for example \cite[p.1330]{P}). Thus $G_v=S_3\times S_4$ or $(C_3\times A_4).2$ and in both cases $\Gamma$ is $(G,3)$-transitive.

Finally, assume that $G_{vw}^{[1]}\not=1$.
Then by \cite{W2}, $G_{vw}^{[1]}$ is a 3-group,
\[G_v=3^2{:}\GL(2,3),\ \mbox{or}\ [3^5]{:}\GL(2,3),\]
and $\Gamma$ is $(G,4)$-transitive or $(G,7)$-transitive, respectively.
\end{proof}

\vskip0.1in
\noindent{\bf Remarks:}
\begin{enumerate}
 \item In case (2) of Lemma \ref{val=4} we have $G_{vw}=S_3\times S_3$ or $G_{vw}=(C_3\times C_3).2$ with respectively $G_v^{[1]}=S_3$ or $C_3$.
 \item The stabiliser $G_v=3^2{:}\GL(2,3)$ is a parabolic subgroup of $\PGL(3,3)$, while
the stabiliser $G_v=[3^5]{:}\GL(2,3)$ is a parabolic subgroup of the exceptional group $G_2(3)$ of Lie type. In both cases $G_v^{\Gamma(v)}\cong\PGL(2,3)\cong S_4$ and $(G_v^{[1]})^{\Gamma(w)}\cong S_3$.
\end{enumerate}

\vskip0.1in

\begin{lemma}\label{val=4-2}
Assume that $\Gamma$ is of valency $4$, let $v \in \VG$, and let $G \le \Aut(\G)$.
Then $\Gamma$ is $G$-star-transitive and $G$-st(edge)-transitive if and only if 
one of the following is true.
\begin{itemize}
\item[(1)] $\Gamma$ is $(G,3)$-transitive, and $G_v=S_4\times S_3$;

\item[(2)] $\Gamma$ is $(G,4)$-transitive, and $G_v=3^2{:}\GL(2,3)$;

\item[(3)] $\Gamma$ is $(G,7)$-transitive, and $G_v=[3^5]{:}\GL(2,3)$.
\end{itemize}
\end{lemma}
\begin{proof}
By Corollary \ref{c:3-arc transitive} and Lemma \ref{lem:edgeaction}, if $\Gamma$ is $G$-star-transitive and $G$-st(edge)-transitive  then $\Gamma$ is $(G,3)$-arc transitive and $G_{vw}$ induces $S_3\times S_3$ on $(\Gamma(v)\cup\Gamma(w))\backslash \{v,w\}$. This rules out case (1) of Lemma \ref{val=4} and case (2) where $G_v=(A_4\times C_3).2$.  By Lemma \ref{lem:stedge} and  the remarks following Lemma \ref{val=4}, it follows that the case where $G_v=S_4\times S_3$ is $G$-st(edge)-transitive. It remains to prove that the $(G,4)$-arc-transitive graphs given in Lemma~\ref{val=4} $(3)$ and $(4)$ are $G$-st(edge)-transitive.

Suppose that $G_v=3^2{:}\GL(2,3)$. Since $G_v^{\Gamma(v)}\cong G_v/G_v^{[1]}$ is a transitive subgroup of $S_4$, the normal subgroup structure of $G_v$ implies that $G_v^{[1]}=3^2{:}2$, and $G_{vw}=3^2{:}2.S_3$.  Since $G_v^{[1]} \neq 1$,  it follows that there exists $u\in\Gamma(v)$ such that  $G_v^{[1]}$ acts on  $\Gamma(u)$ non-trivially; otherwise $G_v^{[1}=G_v^{[2]}$, and it follows from the connectivity of $\Gamma$ that $G_v^{[1]}$ fixes all vertices, contradicting $G_v^{[1]}\neq 1$. Since $G_v$ is transitive on $\Gamma(v)$ we deduce that  $G_v^{[1]}$ acts non-trivially on $\Gamma(w)$. Hence $1 \neq (G_v^{[1]})^{\Gamma(w)} \lhd G_{vw}^{\Gamma(w)} = S_3$. Recall also that in this case $G_{vw}^{[1]}$ is a 3-group and so $(G_v^{[1]})^{\Gamma(w)}$ has even order.  Thus $G_v^{[1]}/G_{vw}^{[1]}=S_3$ and so  $G_{vw}^{[1]}=C_3$.
Then we have
\[3.2.S_3\cong G_{vw}/G_{vw}^{[1]}\cong G_{vw}^{\Gamma(v)\cup\Gamma(w)}\le
G_{vw}^{\Gamma(v)}\times G_{vw}^{\Gamma(w)}\cong S_3\times S_3.\]
Therefore, $G_{vw}^{\Gamma(v)\cup\Gamma(w)}\cong S_3\times S_3$, and so by Lemma \ref{lem:stedge},
$\Gamma$ is $G$-st(edge)-transitive.

Next suppose that  $G_v=[3^5]{:}\GL(2,3)$.  Then the normal subgroup structure of $G_v$ implies that $G_v^{[1]}=[3^5].2$ and $G_{vw}=[3^5].2.S_3$. Arguing as in the previous case we deduce again that $G_v^{[1]}/G_{vw}^{[1]}\cong S_3$, and so $G_{vw}^{[1]}=[3^4]$. A similar argument to the previous case  reveals that  $G_{vw}^{\Gamma(v)\cup\Gamma(w)}\cong S_3\times S_3$ and so by Lemma \ref{lem:stedge}, $\Gamma$ is st(edge)-transitive.
\end{proof}

Next, we consider the case of valency at least 5.

\begin{lemma}\label{val>4}
Suppose that $\Gamma$ is of valency $r\geq 5$, let $v \in \VG$, and let $G \le \Aut(\G)$. Then $\Gamma$ is $G$-star-transitive if and only if one of the following holds.

\begin{enumerate}[$(1)$]
\item $\Gamma$ is $(G,2)$-transitive and $G_v=S_r$;

\item $\Gamma$ is $(G,3)$-transitive, $r=7$,  $G_v=S_7$ and $G_{\{u,v\}}=\Aut(A_6)$;

\item $\Gamma$ is $(G,3)$-transitive, and $G_v=S_r\times S_{r-1}$ or $(A_r\times A_{r-1}).2$ with the element of order 2 inducing a nontrivial outer automorphism of both $A_r$ and $A_{r-1}$;

\item $r=5$, $\Gamma$ is $(G,4)$-transitive and $G_v=[4^2]{:}\GammaL(2,4)$;

\item $r=5$, $\Gamma$ is $(G,5)$-transitive and  $G_v=[4^3]{:}\GammaL(2,4)$.
\end{enumerate}
Moreover, $\Gamma$ is $G$-star-transitive and $G$-st(edge)-transitive if and only if
$\Gamma$ is $(G,3)$-transitive and $G_v=S_r\times S_{r-1}$.
\end{lemma}
\begin{proof}
Suppose that $\G$ and $G$ satisfy one of $(1)$  - $(5)$. Then $G_v^{\Gamma(v)}$ is 2-transitive. Except when both $r=6$ and case (3) holds, the only 2-transitive quotient group of $G_v$ on $r$ points is $S_r$ and so $G$ is locally fully symmetric. Thus by Lemma \ref{l:suff star}, $\Gamma$ is $G$-star-transitive. It remains to consider case (3) when $r=6$. Here $G_v$ has $S_6$ and $S_5$ as 2-transitive factor groups of degree 6. If $G_v^{\Gamma(v)}=S_5$ acting 2-transitively on 6 points it follows that for a 2-arc $uvw$ we have  $G_{uvw}=S_6\times C_4$, which does not have a transitive action of degree 5, contradicting $G$ being 3-arc transitive. Thus $G_v^{\Gamma(v)}=S_6$ and so $\Gamma$ is $G$-star-transitive in this case as well.

Conversely, suppose that $\Gamma$ is $G$-star-transitive.
If $G_v^{[1]}=1$, then $G_v\cong G_v^{\Gamma(v)}\cong S_r$ and so $\Gamma$ is $(G,2)$-arc-transitive. For $u\in\Gamma(v)$ we have $G_{uv}=S_{r-1}$ which acts faithfully on both $\Gamma(u)\backslash\{v\}$ and $\Gamma(v)\backslash\{u\}$. For $r-1\neq 6$ this implies that the actions are equivalent, that is, the stabiliser of a vertex in one action fixes a vertex in the other. So for a 2-arc $wvu$ we have that $G_{wvu}$ fixes an element of $\Gamma(u)\backslash\{v\}$ and hence $\Gamma$ is not $(G,3)$-arc-transitive.  For $r=7$, $S_{r-1}$ has two inequivalent actions of degree 6 that are interchanged by an outer automorphism. The stabiliser of a point in one action is transitive in the other action and so $\Gamma$ will be $(G,3)$-arc-transitive if and only if $G_{\{u,v\}}=\Aut(S_6)$. Moreover, $\Gamma$ is not $(G,4)$-arc-transitive in this case as the stabiliser of a 3-arc is $C_5\rtimes C_4$, which does not have a transitive action of degree 6.

If $G_v^{[1]}\not=1$ and $G_{vw}^{[1]}=1$, then we have
\[1\not=G_v^{[1]}\cong G_v^{[1]}/G_{vw}^{[1]}\cong(G_v^{[1]})^{\Gamma(w)}\lhd G_{vw}^{\Gamma(w)}\cong S_{r-1}.\]
Thus either $G_v^{[1]}=A_{r-1}$ or $S_{r-1}$, or $r=5$ and $G_v^{[1]}=2^2$. The last case is eliminated by \cite[Lemma 5.3]{M}. Thus for  $u\in\Gamma(v)\backslash\{w\}$ we have that $G_{uvw}$ induces either $A_{r-1}$ or $S_{r-1}$ on the set of $r-1$ 3-arcs beginning with $uvw$. Hence $\Gamma$ is $(G,3)$-arc-transitive.   Moreover, since $G_{vw}^{[1]}=1$, it follows from \cite{W2} that  $\Gamma$ is not $(G,4)$-arc-transitive.  If $G_v^{[1]}=S_{r-1}$, since $S_{r-1}$ has no outer automorphisms for $r\neq 7$, it follows that either $G_v=S_r\times S_{r-1}$ or $r=7$ and $G_v=(A_7\times S_6).2$ with elements of  $G_v\backslash (A_7\times S_6)$ inducing an outer automorphism of both $S_6$ and $A_7$.  Suppose that we have the latter case. Then $G_{vw}=(A_6\times S_6).2$. However, in such a group $G_v^{[1]}$ is a characteristic subgroup of $G_{vw}$ as it is the only normal subgroup isomorphic to $S_6$. Thus $G_v^{[1]}$ is normalised in $G_{\{v,w\}}$ and hence normal in $\langle G_v,G_{\{v,w\}}\rangle=G$. Hence $G_v$ contains a nontrivial normal subgroup of $G$,  contradicting the action on $V\Gamma$ being faithful. Thus if $G_v^{[1]}=S_{r-1}$ then $G_v=S_r\times S_{r-1}$.

 If $G_v^{[1]}=A_{r-1}$ then $G_v=S_r\times A_{r-1}$ or $(A_r\times A_{r-1}).2$.  However, in the first case we can argue as above to show that $G_v^{[1]}\lhd G$ again yielding a contradiction. (In this case $G_v^{[1]}$ is characteristic in $G_{vw}$ as it is the only normal subgroup isomorphic to $A_{r-1}$ not contained in an $S_{r-1}$.) Similarly in the second case, elements of $G_{vw}\backslash (A_{r-1}\times A_{r-1})$ must induce nontrivial outer automorphisms of both normal subgroups isomorphic to $A_{r-1}$ and so $G_v= (A_r\times A_{r-1}).2$ with $G_v/A_r\cong S_{r-1}$ and $G_v/A_{r-1}\cong S_r$.

Next, assume that $G_{vw}^{[1]}\not=1$.
By Theorem~\ref{t:arc kernel}, $G_{vw}^{\Gamma(v)} \cong  S_{r-1}$ is $p$-local.
Thus, $r=5$, and $p=2$.
Further, by \cite{W2}, either $\Gamma$ is $(G,4)$-transitive, and
$G_v=[4^2]{:}\GL(2,4)$ or $[4^2]{:}\GammaL(2,4)$ (that is, the maximal parabolics in $\PGL(3,4)$ or $\PGaL(3,4)$ respectively),
or $\Gamma$ is $(G,5)$-transitive, and 
$G_v=[4^3]{:}\GL(2,4)$ or $[4^3]{:}\GammaL(2,4)$ (that is, the maximal parabolics in $\PSp(4,4)$ or  $\PGaSp(4,4)$ respectively). Since neither $[4^2]{:}\GL(2,4)$ nor $[4^3]{:}\GL(2,4)$ have $S_5$ as a quotient group they cannot occur.

We also have to establish which graphs in the list are st(edge)-transitive. 
If $G_v=S_r\times S_{r-1}$ in case $(3)$ 
then $G_{vw}=S_{r-1}\times S_{r-1}$ acts faithfully on $(\Gamma(v)\cup\Gamma(w))\backslash\{v,w\}$ and so the arc-transitivity of $\Gamma$ together with Lemma \ref{lem:stedge} implies that $\Gamma$ is $G$-st(edge)-transitive. In all other graphs in cases $(1)$, $(2)$ and $(3)$, $G_v$ is too small to satisfy the necessary condition in Lemma~\ref{lem:edgeaction}. In cases (3) and (4), Lemma \ref{lem:stedge} implies that to be $G$-st(edge)-transitive we must have that $G_v^{[1]}$ has $S_4$ as a quotient group. However, a simple \textsf{GAP} \cite{gap} computation shows that this is not the case.
\end{proof}

\vskip0.1in
\noindent{\bf Remark:} 
The stabiliser $G_v=[4^2]{:}\GammaL(2,4)$ is a parabolic subgroup in $\PGaL(3,4)$ while the stabiliser $G_v=[4^3]{:}\GammaL(2,4)$ is a parabolic subgroup in $\PGaSp(4,4)$. In both cases $G_v^{\Gamma(v)}\cong \PGaL(2,4)\cong S_5$.

\vskip0.1in
\noindent{\bf Remark:} The Hoffman--Singleton graph with automorphism group $\PSU(3,5).2$ on 50 vertices is an example of a graph in case (2).

\vskip0.1in
Combining the lemmas in this section, we obtain a characterisation of graphs which are vertex-transitive,
star-transitive and st(edge)-transitive, as stated in Theorem~\ref{v-star-st(edge)-t}.

\section{The vertex intransitive case}
\label{s:vertex intran}

In this section, we study connected star-transitive and st(edge)-transitive graphs $\Gamma$ with vertex-intransitive automorphism groups. Such graphs must be bipartite, of bivalency $\{\ell,r\}$ for some $\ell \ne r$. 
The analogue of Theorem~\ref{t:arc kernel} was proved by  van Bon~\cite{vB1}.

\begin{theorem}
\label{t:vB}
Let $\Gamma$ be a connected graph and $G\leqslant \Aut(\Gamma)$ such that $G_v^{\Gamma(v)}$ is primitive for all vertices $v$.  Then for an edge $\{v,w\}$ either:
\begin{itemize}
\item[(i)] $G_{vw}^{[1]}$ is a $p$-group; or
\item[(ii)] (possibly after interchanging $v$ and $w$), $G_{vw}^{[1]}=G_v^{[2]}$, and $G_w^{[2]}=G_w^{[3]}$ is a $p$-group.
\end{itemize}
\end{theorem}

Assume that $G\le\Aut(\Gamma)$, and $\G$ is $G$-star-transitive and $G$-st(edge)-transitive.
Let $uvwx$ be a $3$-arc of $\Gamma$, and suppose that $|\Gamma(v)|=r$ and $|\Gamma(w)|=\ell$. We note that $G_{vw}^{[1]}$ acts on both $\Gamma(u)$ and $\Gamma(x)$.

\begin{lemma}\label{3,5}
Assume that $G_{vw}^{[1]}$ acts non-trivially on both $\Gamma(u)$ and $\Gamma(x)$.
Then $\{\ell,r\}=\{3,5\}$.
\end{lemma}
\begin{proof}
Since $G_{vw}^{[1]}$ acts non-trivially on $\Gamma(u)$ it follows that $G_{vw}^{[1]}\neq G_v^{[2]}$ and so Theorem~\ref{t:vB} implies that $G_{vw}^{[1]}$ is a $p$-group. Since 
$G_{vw}^{[1]}\lhd G_w^{[1]}\lhd G_{wx}$ and $G_{vw}^{[1]}$ acts non-trivially on $\Gamma(x)$, we have
\[1\not=(G_{vw}^{[1]})^{\Gamma(x)}\lhd (G_w^{[1]})^{\Gamma(x)}\lhd G_{wx}^{\Gamma(x)}.\]
Thus, $G_{wx}^{\Gamma(x)}\cong S_{r-1}$ has a subnormal $p$-subgroup, and similarly,  so does $G_{uv}^{\Gamma(u)}\cong S_{\ell-1}$.
Hence $r,\ell\le 5$.
If $r=4$, then $G_{vw}^{\Gamma(v)}\cong S_3$ and $p=3$.
As $\ell\not=r$, we have $\ell=3$ or 5, and thus $G_{vw}^{\Gamma(w)}=S_2$ or $S_4$,
which do not have subnormal $3$-subgroups. This is a contradiction.
Thus, $r\not=4$, and similarly, $\ell\not=4$. Hence $\{\ell,r\}=\{3,5\}$.
\end{proof}

\begin{example}
\label{eg:hermitian}
Let $\Gamma$ be the point-line incidence graph of the generalised quadrangle of order $(2,4)$ arising from a nondegenerate Hermitian form on a 4-dimensional vector space over $\GF(4)$. That is, the points are the totally isotropic 1-spaces, the lines are the totally isotropic 2-spaces, and a 1-space and 2-space are incident if one is contained in the other. Let $G=\PGammaU(4,2)=\PSU(4,2).2$, the full automorphism group of $\Gamma$, and let $(w,v)$ be an incident point-line pair. Then $G_v=2^4{:}S_5$, and $G_w=2.(A_4\times A_4).2^2$
with $G_v\cap G_w=2^4.S_4$.
Then $|G_v:G_v\cap G_w|=5$ and $|G_w:G_v\cap G_w|=3$, and $\Gamma$ is locally $(G,4)$-arc-transitive
of bivalency $\{3,5\}$. It follows from Lemma \ref{l:suff star} that $\Gamma$ is star-transitive.
We have $G_v^{[1]}=2^4$, and $G_{vw}^{[1]}=2^3$. Thus $G_{vw}/G_{vw}^{[1]}=S_4\times S_2$, and so by Lemma \ref{lem:stedge}, $\Gamma$ is st(edge)-transitive.
\end{example}
\vskip0.1in

\begin{lemma}\label{[2]=1}
Assume that $G_w^{[1]}\not=1$ and $G_w^{[2]}=1$.
Then one of the following holds:
\begin{itemize}
\item[(i)] $G_{vw}^{[1]}=1$, $G_v=S_r\times S_{\ell-1}$ and $G_w=S_{\ell}\times S_{r-1}$;

\item[(ii)] $(A_{r-1})^{\ell-1}\leqslant G_{vw}^{[1]}\leqslant (S_{r-1})^{\ell-1}$,  
\[\begin{array}{rllll}
(A_r\times (A_{r-1})^{\ell-1}).2.S_{\ell-1}&\leqslant &G_v&\leqslant&S_r\times (S_{r-1}\Wr S_{\ell-1}), \mbox{and}\\

(A_{r-1})^{\ell}.2.S_\ell& \leqslant& G_w &\leqslant &S_{r-1}\Wr S_\ell; \textrm{ or}
\end{array}\]

\item[(iii)] 
$|\Gamma(v)|\le 5$.

\end{itemize}
\end{lemma}
\begin{proof}
Since $G_w^{[2]}=1$, we have
\[\begin{array}{l}
G_w\cong G_w/G_w^{[2]}\cong G_w^{\Gamma_1(w)\cup\Gamma_2(w)}\le 
G_{vw}^{\Gamma(v)}\Wr G_w^{\Gamma(w)}\cong S_{r-1}\Wr S_{\ell}, \text{ and }\\

G_{vw}\cong G_{vw}/G_w^{[2]}\cong G_{vw}^{\Gamma_1(w)\cup\Gamma_2(w)}\le 
(S_{r-1}\Wr S_{\ell-1})\times S_{r-1}\\
\end{array}\]
Let $\Gamma(w)=\{v_1,\ldots,v_\ell\}$ with $v_1=v$. For each $i\in\{1,\ldots, \ell\}$, let $B_i$ be the subgroup of $\Sym(\Gamma_1(w)\cup\Gamma_2(w))$ consisting of all permutations that fix each vertex
in $\Gamma(w)$, act trivially on $\Gamma(v_j)$ for $j\neq i$ and induce an element of $S_{r-1}$ on $\Gamma(v_i)\backslash\{w\}$. Then $B_i\cong S_{r-1}$. Also, let $C$ be a subgroup of $\Sym(\Gamma_1(w)\cup\Gamma_2(w))$ isomorphic to $S_\ell$ and such that $C$ acts faithfully on $\Gamma(w)$. 
In particular, if $g\in C$ and $v_i^g=v_j$ then $B_i^g=B_j$.
Then we can identify $G_w$ with a subgroup of $X:=(B_1\times B_2\times\cdots\times B_\ell)\rtimes C$.
Thus $G_{vw}=G_w\cap ( B_1\times ((B_2\times \cdots \times B_\ell)\rtimes C_v))$, where $C_v\cong S_{\ell-1}$ and $G_w^{[1]}=G_w\cap (B_1\times\cdots\times B_\ell)$. 

Let $\rho$ be the projection of $X$ onto $C$ and for each $i=1,2,\ldots, \ell$, let $\pi_i$ be the projection of $B_1\times\cdots\times B_\ell$ onto $B_i$. Since $G$ is st(edge)-transitive, by Lemma \ref{lem:edgeaction} we have that $G_w^{\Gamma(w)}=S_\ell$ and so $\rho(G_w)=C$. Moreover, 
$G_{vw}^{\Gamma(v)\cup\Gamma(w)}\cong S_{r-1}\times S_{\ell-1}$ and so $\rho(G_{vw})=C_v\cong S_{\ell-1}$ and $\pi_1( \ker(\rho)\cap G_{vw})=B_1\cong S_{r-1}$. Since $\ker\rho\cap G_w=G_w^{[1]}$ it follows that $\pi_1(G_w^{[1]})\cong S_{r-1}$. Since $G_w$ acts transitively on the $\ell$ factors $B_1,\ldots,B_{\ell}$ it follows that $\pi_i(G_w^{[1]})\cong\pi_j(G_w^{[1]})\cong S_{r-1}$ for distinct $i$ and $j$. 

Suppose that $r\geq 6$. Then $A_{r-1}$ is a nonabelian simple group and since $G_w$ acts primitively on the $\ell$ factors $B_1,\ldots,B_{\ell}$ it follows from \cite[p. 328]{scott} that either $G_{w}^{[1]}\cong \pi_i(G_{w}^{[1]})$ for each $i$, or $(A_{r-1})^\ell\cong \soc(B_1)\times\cdots\times\soc(B_\ell)\leqslant G_w^{[1]}$.  Since $G_{vw}^{[1]}\leqslant G_w^{[1]}$ and $\pi_1(G_{vw}^{[1]})=1$ these two cases correspond to $G_{vw}^{[1]}=1$ and $G_{vw}^{[1]}\neq 1$ respectively.

Suppose first that $G_{vw}^{[1]}=1$. Then $G_{vw}\cong S_{r-1}\times S_{\ell-1}$. As $\ker(\rho)\cap G_w\leqslant G_{vw}$ we have $\ker(\rho)\cap G_w\cong S_{r-1}$. Moreover, as $\rho(G_w)=S_\ell$ we have $G_w/(\ker(\rho)\cap G_w)\cong S_{\ell}$ with $\rho(G_{vw})=S_{\ell -1}$ and $G_{vw}\cap \ker\rho\cong S_{r-1}$. Thus either $G_w\cong S_{r-1}\times S_\ell$, or $G_w\cong (S_{r-1}\times A_{\ell}).2$ with each element of $G_w$ not in $S_{r-1}\times A_{\ell}$ inducing a nontrivial automorphism of both $S_{r-1}$ and $A_\ell$. Since $G_w$ contains  $G_{vw}\cong S_{r-1}\times S_{\ell-1}$, which has an element of $S_{\ell-1}$ not in $A_\ell$ that centralises $S_{r-1}$, the second case is not possible. Thus $G_w\cong S_{r-1}\times S_\ell$.  Similarly, $G_v^{[1]}=S_{\ell -1}$ and $G_v^{\Gamma(v)}\cong S_r$ and arguing as in the $G_w$ case we deduce that $G_v=S_r\times S_{\ell-1}$. Hence we are in case (i).

Assume now that $G_{vw}^{[1]}\neq 1$. Then, as we have already seen, $\soc(B_1)\times\cdots\times\soc(B_\ell)\leqslant G_w^{[1]}$.  Since $\pi_1(G_w^{[1]})=S_{r-1}$ and $\pi_i(G_w^{[1]})\cong \pi_j(G_w^{[1]})$ for all distinct $i$ and $j$, it follows that $G_w^{[1]}$ contains a subgroup $N$ isomorphic to $(A_{r-1})^\ell.2$ such that $\pi_i(N)\cong S_{r-1}$ for all $i$. Thus
$$(A_{r-1}^{\ell}).2.S_\ell \leqslant G_w \leqslant S_{r-1}\Wr S_\ell$$
Moreover, for the arc stabiliser we have
$$(A_{r-1})^{\ell}.2.S_{\ell-1} \leqslant G_{vw}\leqslant S_{r-1}\times (S_{r-1}\Wr S_{\ell-1})$$
with $(A_{r-1})^{\ell-1}\leqslant G_{vw}^{[1]}\leqslant (S_{r-1})^{\ell-1}$.

Consider the group $\overline{G}_w=G_w/(\soc(B_1)\times\cdots\times\soc(B_\ell))$. Then $2.S_{\ell}\leqslant \overline{G}_w\leqslant 2^\ell.S_\ell$. Hence $\overline{G}_w=2^m.S_\ell$ where the $2^m$ is a submodule of the permutation module for $S_\ell$. This implies that $G_{w}=(A_{r-1})^\ell.2^m.S_\ell$. Hence $G_{vw}=(A_{r-1})^\ell.2^m.S_{\ell-1}$ and $G_v^{[1]}=(A_{r-1})^{\ell-1}.2^{m-1}.S_{\ell-1}$, as $G_{vw}/G_{v}^{[1]}\cong G_{vw}^{\Gamma(v)}\cong S_{r-1}$. Observe that $G_v^{[1]}$ has a unique minimal normal subgroup $N$ and $N\cong (A_{r-1})^{\ell-1}$. Thus $N\lhd G_v$. Since $G_v/G_v^{[1]}\cong S_r$ and $G_v^{[1]}$ already induces $S_{\ell-1}$ on the $\ell-1$ distinct simple factors of $N$, it follows that $G_v$ contains a normal subgroup isomorphic to  $A_r\times (A_{r-1})^{\ell-1}$. We then deduce that $G_v=(A_r\times (A_{r-1})^{\ell-1}).2^m.S_\ell$ and in particular part (ii) holds.
\end{proof}

By Lemma \ref{[2]=1}  we may thus assume that $G_v^{[2]}\neq 1$ and $G_w^{[2]}\neq 1$. Moreover, by Lemma \ref{3,5} we may assume that $G_{vw}^{[1]}$ acts trivially on $\Gamma(u)$ for some neighbour $u$ of $v$. Since $G_{vw}$ is transitive on $\Gamma(v)\setminus\{w\}$ and normalises $G_{vw}^{[1]}$,
we conclude that $G_{vw}^{[1]}$ fixes all vertices in $\Gamma_2(v)$. In particular, $G_{vw}^{[1]}=G_v^{[2]}$.

\begin{lemma}\label{[2]not=1}
Assume that $G_v^{[2]}\neq 1\neq G_w^{[2]}$ and $G_{vw}^{[1]}=G_v^{[2]}$ fixes all vertices in $\Gamma_2(v)$. Then $r\leq 5$ and either $G_w^{[2]}$ or $G_v^{[2]}$ is a $p$-group with $p=2$ or $3$.
\end{lemma}
\begin{proof}
By Theorem~\ref{t:vB}, either $G_v^{[2]}$ is a $p$-group, or  $G_w^{[2]}=G_w^{[3]}$ is a $p$-group.

Suppose first that $G_v^{[2]}$ is a $p$-group. Then 
$$1\neq G_v^{[2]}\lhd G_w^{[1]}\lhd G_{wx}$$
Suppose that  $G_v^{[2]}$ acts trivially on $\Gamma(x)$.
Then as $G_v^{[2]}\lhd G_{vw}$ and $G_{vw}$ acts transitively on $\Gamma(w)$ 
it follows that $G_v^{[2]}$ acts trivially on $\Gamma_2(w)$.  
Thus, $G_v^{[2]}\le G_w^{[2]}\le G_{vw}^{[1]}=G_v^{[2]}$, and so $G_v^{[2]}=G_w^{[2]}$.
Since $G$ is edge-transitive on $\Gamma$, we conclude that $G_v^{[2]}$ fixes all vertices of $\Gamma$, which is a contradiction.
%
%
 Hence $G_v^{[2]}$ acts nontrivially on $\Gamma(x)$ and so $G_{wx}^{\Gamma(x)}\cong S_{r-1}$ has a nontrivial normal $p$-subgroup. Thus $r\leq 5$ and $p=2$ or $3$.

Next suppose that $G_w^{[2]}=G_w^{[3]}$ is a $p$-group.
Pick $y\in\Gamma(u)\setminus\{v\}$.
Then 
\[1\not=G_w^{[2]}=G_w^{[3]}\lhd G_{uv}^{[1]}\lhd G_u^{[1]}\lhd G_{uy}.\]
Suppose that $G_w^{[3]}$ acts trivially on $\Gamma(y)$.
Since $G_{wvu}$ normalises $G_w^{[3]}$ and is transitive on $\Gamma(u)\backslash\{v\}$,
it follows that $G_w^{[3]}$ is trivial on $\Gamma(x)$ for all $x\in \Gamma(u)$,
and hence on $\Gamma_2(u)$.
Further, $G_{wv}$ normalises $G_w^{[3]}$ and is transitive on $\Gamma(v)\setminus\{w\}$,
and it follows that $G_w^{[3]}$ acts trivially on $\Gamma_2(z)$ for all $z\in\Gamma(v)$.
Thus, $G_w^{[3]}$ fixes all vertices in $\Gamma_3(v)$, and $G_w^{[3]}\le G_v^{[3]}$.
Since $G$ is edge-transitive on $\Gamma$, we conclude that $G_w^{[3]}=G_w^{[4]}$.
Hence we have 
\[1\not=(G_w^{[3]})^{\Gamma(y)}\lhd (G_{uv}^{[1]})^{\Gamma(y)}\lhd
(G_u^{[1]})^{\Gamma(y)}\lhd G_{uy}^{\Gamma(y)}\cong S_{r-1}.\]
Thus, $S_{r-1}$ has a subnormal $p$-subgroup and so $r\le 5$, and $p=2$ or 3.
\end{proof}

Combining the lemmas of this section, we obtain the characterisation of vertex-intransitive,
star-transitive and st(edge)-transitive graphs, as stated in Theorem~\ref{vtx-intrans}.

\section{Examples}\label{s:examples}

Lazarovich \cite{L} provides a list of examples of graphs $L$ which are star-transitive and st(edge)-transitive, as well as a discussion of earlier work on uniqueness of $(k,L)$-complexes.  In this section we use our results above to expand on the examples and discussion in \cite{L}, and to give new infinite families of examples in both the vertex-transitive  and vertex-intransitive cases.

Recall that we denote by $\XkL$ the collection of all simply-connected $(k,L)$-complexes and that we assume $L$ is finite and connected.

\subsection{Special cases}

\subsubsection{Cycles}

The only finite, connected, $2$-regular graph is the cycle on $n$ vertices $C_n$.  For $n \geq 3$, as observed in \cite{L}, the graph $C_n$ is star-transitive and st(edge)-transitive.  If $(k,n) \in \{ (3,6), (4,4), (6,3)\}$ and the polygons are metrised as regular Euclidean $k$-gons, then the unique simply-connected $(k,C_n)$-complex is the tessellation of the Euclidean plane by regular Euclidean $k$-gons.  If $k \geq k'$ and $n > n'$, or $k > k'$ and $n \geq n'$, where $(k',n') \in \{ (3,6), (4,4), (6,3)\}$, then the unique simply-connected $(k,C_n)$-complex is combinatorially isomorphic to the tessellation of the hyperbolic plane by regular hyperbolic $k$-gons with vertex angles $\frac{2\pi}{n}$. 

\subsubsection{Cubic graphs}

Let $k \geq 3$ be an integer and let $L$ be a finite, connected, cubic graph, such that the pair $(k,L)$ satisfies the Gromov Link Condition.  In \cite{Sw}, \Swiatkowski\ proved that if $k \geq 4$ and $L$ is $3$-arc transitive, then there exists a unique simply-connected $(k,L)$-complex, while if $k = 3$ and $L$ is $3$-arc transitive, or $k \geq 3$ and $L$ is not $3$-arc transitive, $|\XkL| > 1$.  (In fact, \Swiatkowski\ proved that in these latter cases $\XkL$ is uncountable; compare \cite[Theorem B]{L}.)

It follows from results of Tutte \cite{tutte47, tutte59} that, as observed in \cite{L}, a connected cubic graph which is $3$-arc transitive is both star-transitive and st(edge)-transitive.  Conversely, if a connected cubic graph $L$ is star-transitive and st(edge)-transitive then $L$ is $3$-arc transitive, by Corollary \ref{c:3-arc transitive} above.  Together with the main results of \cite{L}, this recovers the results of \Swiatkowski\  from \cite{Sw} when $k \geq 4$ is even.

The collection of $3$-arc transitive cubic graphs is too large to classify.  For examples of such graphs, see for instance \cite{GR}.

\subsubsection{Complete bipartite graphs}

It is observed in \cite{L} that the complete bipartite graph $K_{m,n}$ is star-transitive and st(edge)-transitive.   The unique simply-connected $(4,K_{m,n})$-complex is the product of an $m$-valent and an $n$-valent tree \cite{Wise96}.  If $k > 4$, $m,n \geq 2$ and either $k$ is even or $m = n$ the unique simply-connected $(k,K_{m,n})$-complex is isomorphic to Bourdon's building, which is a $2$-dimensional hyperbolic building (see \cite{B1}).  If $k$ is odd and $m \ne n$, there is no $(k,L)$-complex.

By Lemma \ref{lem:valency one} above, the only finite, connected graphs with a vertex of valency one that are star-transitive and st(edge)-transitive are the complete bipartite graphs $K_{1,n}$.  If $n \geq 2$ and $k \geq 4$ is even then the unique simply-connected $(k,K_{1,n})$-complex $X$ is a ``tree" of $k$-gons, with alternating edges of each $k$-gon contained in either a unique $k$-gon or $n$ distinct $k$-gons.

\subsubsection{Odd graphs}\label{s:odd}

For $n \geq 2$ the \emph{Odd graph} $O_{n+1}$ is defined to have vertex set the $n$-element subsets of $\{ 1, 2, \ldots, 2n+1\}$, with two vertices being adjacent if the corresponding $n$-sets are disjoint.  Thus the graph is $(n+1)$-valent.  The Petersen graph is the case $n = 2$.  

As noted in \cite{L}, the Odd graphs are star-transitive and st(edge)-transitive.  Their vertex stabilisers are $S_{n} \times S_{n+1}$, and their edge-stabilisers are $S_n \Wr S_2$.  The girth of the Petersen graph $O_3$ is $5$ and for all $n \geq 3$ the girth of $O_{n+1}$ is $6$. 

It was proved by Praeger  \cite[Theorem 4]{CEP} that if $\Gamma$ is a graph of valency $r$ with  $A_r\leqslant G_v^{\Gamma(v)}$ and  $G$ primitive on $V\Gamma$ then either $|\Gamma_3(v)|=r(r-1)^2$, that is girth at least 7, or $\Gamma$ is the Odd graph on $2r-1$ points.  Thus by Lemma \ref{l:locally fully symmetric}, if $\Gamma$ is $G$-star-transitive with $G$ primitive on vertices then either $\Gamma$ has girth at least 7 or is an Odd graph.  Vertex-transitive graphs of girth 5 and with $G_v^{\Gamma(v)}=S_r$ were investigated by Ivanov, who showed that  that such graphs are either the Petersen graph, the Hoffman--Singleton graph, the double cover of the Clebsch graph or each 2-arc in $\Gamma$ is contained in a unique cycle of length 5 \cite[Lemma 5.6]{I}. Combining \cite[Lemma 5.4]{I} and Theorem \ref{v-star-st(edge)-t} we deduce that the only vertex-transitive graph of girth 5 that is both star-transitive and st(edge)-transitive is the Petersen graph.

We note that any finite cover of an Odd graph for which all automorphisms lift will also be star-transitive and st(edge)-transitive.  Thus for every valency $n+1 \geq 3$ we obtain an infinite family of $(n+1)$-regular graphs which are star-transitive and st(edge)-transitive.

\subsubsection{Spherical buildings}

Let $\mathfrak{L}$ be a simple Lie algebra of rank $2$ over $\mC$.  So
$\mathfrak{L}$ is of type $A_2$, $B_2 = C_2$ or $G_2$.  We
denote by $\mathfrak{L}(q)$ the (untwisted) Chevalley group of type
$\mathfrak{L}$ over the field $\GF(q)$.  Denote by $L$ the spherical
building associated to the group $\mathfrak{L}(q)$.  Assume that $k
\geq 5$.  Then by results
of Haglund~\cite{H}, in the following cases,
there is a unique simply-connected $(k,L)$-complex:
\begin{enumerate}
\item $\mathfrak{L}$ is of type $A_2$ and $q \in \{2,3\}$;
\item $\mathfrak{L}$ is of type $B_2$ and $q = 2$;
\item $\mathfrak{L}$ is of type $G_2$ and $q = 3$.
\end{enumerate}
In each of these cases, the $(k,L)$-complex obtained may be metrised
as a $2$-dimensional hyperbolic building.
(For larger values of $q$, Haglund also established the uniqueness of
$(k,L)$-complexes which satisfy some additional local conditions,
called
local reflexivity and constant holonomy.)

The local actions for the rank 2 buildings arising from finite groups of Lie type are given in  the first five columns of Table \ref{tab:localbuild}.  The full automorphism group $G$ in each case can be found in \cite[Chapter 4]{vM} and the local actions of $\soc(G)$ are given in \cite[Theorem 8.4.1 and Table 8.1]{vM}. The local actions for $G$ can then be deduced. By determining the values of $q$ for which the local action induces the full symmetric group we can deduce the information in the locally fully symmetric column. Moreover, since the rank 2 buildings for $\PSp(4,q)$ are vertex-transitive if and only if $q$ is even while the rank 2 buildings for $G_2(q)$ are vertex-transitive if and only if $q$ is a power of 3, the star-transitive column follows from Lemma \ref{l:suff star}. The case where $G=\PGammaU(4,2)$ was analysed in Example \ref{eg:hermitian}. The fact that the buildings for $\PSL(3,2)$ and $\PSp(4,2)$ are st(edge)-transitive follow from Corollary \ref{cor:cubic} and the fact that they are 4-arc transitive and 5-arc transitive respectively. The fact that the building for $G_2(3)$ is st(edge)-transitive follows from part (3) of Lemma \ref{val=4-2} and the preceding remark.

 Finally, the fact that the buildings for $\PSL(3,4)$ and $\PSp(4,4)$ are not st(edge)-transitive follows from the fact that for $g\in G_v$ to induce an odd permutation on  $\Gamma(v)$ it must induce a nontrivial field automorphism of the simple group associated with  $G$.  Thus  it is not possible for an element of $G_{vw}$ to induce an odd permutation of $\Gamma(v)$ and an even permutation of $\Gamma(w)$, and so the necessary condition of Lemma \ref{lem:stedge} does not hold.

\begin{center}
 \begin{table}
  \begin{tabular}{l|llllp{1.9cm}ll}
$G$             & $|\Gamma(v)|$ & $G_v^{\Gamma(v)}$ & $|\Gamma(w)|$ &$G_w^{\Gamma(w)}$ & locally fully symmetric & star-transitive& st(edge)-transitive\\ 
\hline 
$\PGaL(3,q)$ & $q+1$ & $\PGaL(2,q)$ & $q+1$ &$\PGaL(2,q)$ & $q=2,3,4$ &  $q=2,3,4$ &$q=2,3$ \\
$\PGaSp(4,q)$ & $q+1$ & $\PGaL(2,q)$ & $q+1$ &  $\PGaL(2,q)$ & $q=2,3,4$ &  $q=2,4$ & $q=2$\\
$\PGammaU(4,q)$ & $q^2+1$ &  $\PGaL(2,q^2)$ & $q+1$ & $\PGaL(2,q)$ & $q=2$& $q=2$& $q=2$\\
$\PGammaU(5,q)$ & $q^2+1$ &  $\PGaL(2,q^2)$ & $q^3+1$& $\PGammaU(3,q)$ & never\\
$\Aut(G_2(q))$        &  $q+1$ &  $\PGaL(2,q)$  & $q+1$  & $\PGaL(2,q)$  & $q=2,3,4$ & $q=3$ & $q=3$   \\
$\Aut({}^3D_4(q))$    & $q+1$ & $\PGaL(2,q)$ & $q^3+1$ & $\PGaL(2,q^3)$ & never \\
$\Aut({}^2F_4(q))$    & $q+1$ & $\PGaL(2,q)$ & $q^2+1$ & $\Aut(Sz(q))$ & never \\ \\

\end{tabular}

\caption{Local actions of rank 2 buildings}
\label{tab:localbuild}
 \end{table}

\end{center}

\subsection{New vertex-transitive examples}
One way to frame the search for examples is in the context of amalgams. Let $G$ be a group with subgroup $H$ and element $g\in G$ that does not normalise $H$ such that $g^2\in H$ and $\langle H,g\rangle=G$. Following Sabidussi \cite{Sa}, we can construct a connected graph $\mathrm{Cos}(G,H,g)$ with vertices the right cosets of $H$ in $G$ and $Hx$ being adjacent to $Hy$ if and only if $xy^{-1}\in HgH$. The group $G$ acts by right multiplication as an arc-transitive group of automorphisms of $\mathrm{Cos}(G,H,g)$ with the stabiliser in $G$ of the vertex corresponding to $H$ being $H$. Conversely, if $\Gamma$ is a connected graph with an arc-transitive group $G$ of automorphisms then for an edge $\{v,w\}$ and $g\in G_{\{v,w\}}\backslash G_{vw}$ we have that $\Gamma$ is isomorphic to $\mathrm{Cos}(G,G_v,g)$. Moreover, note that $G=\langle G_v,G_{\{v,w\}}\rangle$. Thus by knowing the possible $G_v$ and $G_{\{v,w\}}$ for a vertex-transitive, star-transitive and st(edge)-transitive graph, finding examples becomes a search for completions of the amalgam $(G_v,G_{\{v,w\}},G_{vw})$.  

Taking $(G_v,G_{\{v,w\}},G_{vw})=(S_r\times S_{r-1},S_{r-1}\Wr S_2,S_{r-1}\times S_2)$ as in case (1) of Theorem \ref{v-star-st(edge)-t}, the Odd graphs arise when we take $G=S_{2r-1}$ as a completion while the complete bipartite graphs arise when we take $G=S_r\Wr S_2$ as a completion. 

Another example can be constructed with $G=S_{(r-1)^2}$. Let $n=(r-1)^2$ and note that $n=\binom{r}{2}+\binom{r-1}{2}$. Thus letting $\Omega=\{1,\ldots,n\}$ we can identify $\Omega$ with the disjoint union of the set $\Omega_1$ of $2$-subsets of a set $A$ of size $r$ and the set $\Omega_2$ of 2-subsets of a set $B$ of size $r-1$. Then let $H$ be the subgroup of $G$ isomorphic to $S_r\times S_{r-1}$ that has $\Omega_1$ and $\Omega_2$ as its orbits on $\Omega$. Choose $\overline{B}$ to be a subset of $A$ of size $r-1$ and let $\overline{\Omega}_2$ be the subset of $\Omega_1$ consisting of all 2-subsets of $\overline{B}$. Note that $H_{\overline{\Omega}_2}\cong S_{r-1}\times S_{r-1}$ and choosing $g\in G$ to be an element of order 2 that interchanges $\overline{\Omega}_2$ and $\Omega_2$, we have that $\langle H_{\overline{\Omega}_2},g\rangle \cong S_{r-1}\Wr S_2$. Let  $\Gamma=\mathrm{Cos}(G,H,g)$ and let $v$ be the vertex corresponding to the coset $H$ and $w$ the vertex corresponding to the coset $Hg$. Then $G_v=H\cong S_r\times S_{r-1}$, $G_{vw}=H_{\overline{\Omega}_2}\cong S_{r-1}\times S_{r-1}$  and $G_{\{v,w\}}=\langle H_{\overline{\Omega}_2},g\rangle \cong S_{r-1}\Wr S_2$. Thus $G_v^{\Gamma(v)}\cong S_r$ and since $G_{\{v,w\}}$ acts faithfully on $\Gamma(v)\cup\Gamma(w)$ it follows from Lemmas \ref{l:suff star} and \ref{lem:stedge} that $\Gamma$ is 
star-transitive and st(edge)-transitive. It remains to show that $\Gamma$ is connected, that is, we need to show that $\langle H,g\rangle=G$. Let $X=\langle H,g\rangle$. Then $X$ is transitive on $\Omega$ and since $H$ is primitive on each of its orbits it follows that $X$ is primitive on $\Omega$. A transposition on $B$ induces $r-3$ transpositions on $\Omega_2$ and so $X$ contains an element $\sigma$ that moves only $2(r-3)$ elements of $\Omega$. Since $2(r-3) < 2(\sqrt{n}-1)$, \cite[Corollary 3]{LS} implies that $X=A_n$ or $S_n$. If $r$ is even then $\sigma$ is an odd permutation and so $X=S_n=G$. If $r$ is odd then a transposition of $B$ induces $r-2$ transpositions of $\Omega$, and so $H$ also contains an odd permutation in this case.  Thus $X=S_n=G$ for all $r$. We conclude that $\Gamma$ is connected.

\subsection{New vertex-intransitive examples}

We have already seen that the complete bipartite graphs $K_{n,m}$ with $n\neq m$ are examples in the vertex-intransitive case, as is the generalised quadrangle associated with $\PGammaU(4,2)$.

For all natural numbers $m > n \geq 3$, we will construct examples of connected star-transitive and st(edge)-transitive which are $(m,n)$-biregular. (The restriction to $n \geq 3$ is justified by the results of Section \ref{s:small} above on graphs of minimal valency $2$.)

As in the vertex-transitive case, the search for examples can be framed in terms of amalgams. Given a group $G$ and subgroups $L$ and $R$ such that $G=\langle L,R\rangle$ we can construct the bipartite graph $\mathrm{Cos}(G,L,R)$ with vertices the right cosets of $L$ in $G$ and the right cosets of $R$ in $G$ such that $Lx$ is adjacent to $Ry$ if and only if $Lx\cap Ry\neq \varnothing$. Then $G$ acts by right multiplication as an edge-transitive group of automorphisms of $\mathrm{Cos}(G,L,R)$. Conversely, if $\Gamma$ is a connected graph with a group $G$ of automorphisms that is edge-transitive but not vertex-transitive then $\Gamma$ is isomorphic to $\mathrm{Cos}(G,G_v,G_w)$ for some edge $\{v,w\}$. Also $G=\langle G_v,G_w\rangle$. Thus examples of vertex-intransitive, star-transitive and st(edge)-transitive graphs can be found by finding completions of the amalgam $(G_v,G_w,G_v\cap G_w)$.

\subsubsection{Construction from Johnson graphs}\label{s:johnson}

Let $\Gamma=\Gamma_{m,n}$ be the bipartite graph whose vertices are the $m$-subsets and $(m-1)$-subsets of an $n$-set, with two vertices being adjacent if one is contained in the other. Then $G=S_n\leqslant \Aut(\Gamma_{m,n})$. When $n=2m+1$, $\Gamma$ is called the \emph{doubled Odd graph}.  Let $v$ be an $m$-subset. Then $\Gamma(v)$ is the set of $(m-1)$-subsets contained in $v$ and $G_v=S_m\times S_{n-m}$ with $G_v^{\Gamma(v)}=S_m$. Given $w\in\Gamma(v)$ we have $G_w=S_{m-1}\times S_{n-m+1}$ and $\Gamma(w)$ is the set of $m$-subsets containing $w$. Thus $G_w^{\Gamma(w)}=S_{n-m+1}$ and $G_{vw}=S_{m-1}\times S_{n-m}$. Hence by Lemma \ref{l:suff star}, $\Gamma$ is $G$-star-transitive when $m\neq n-m+1$. Moreover, $G_{vw}$ acts faithfully on $\Gamma(v)\cup \Gamma(w)$ so Lemma \ref{lem:edgeaction} implies that $\Gamma$ is also $G$-st(edge)-transitive when $m\neq n-m+1$.

The \emph{Johnson graph} $J(n,m)$ is the graph with vertex set the set of $m$-subsets of an $n$-set such that two $m$-subsets are adjacent if and only if their intersection has size $m-1$. Note that an $(m-1)$-subset defines a maximal clique of $J(n,m)$, namely the set of all $m$-sets containing the given $(m-1)$-set. Let $\mathcal{C}$ be the set of all such maximal cliques and define the graph whose vertices are the vertices of $J(n,m)$ and the maximal cliques in $\mathcal{C}$, with adjacency being the natural inclusion. Then the new graph is isomorphic to $\Gamma_{m,n}$.

\subsubsection{Construction from Hamming graphs}\label{s:hamming}

Another new family of examples are as follows.  This is from Example 4.3 of Giudici--Li--Praeger \cite{GLP1}.  Let $H(k,n)$ be the Hamming graph whose vertex set is the set of ordered $k$-tuples (possibly with repeats) from a set $\Omega$ of size $n$, with two vertices being adjacent if and only if they differ in exactly one coordinate. Let $\Gamma$ be the bipartite graph with vertex set $\Delta_1\cup\Delta_2$, where $\Delta_1$ is the set of vertices of $H(m,n)$ and $\Delta_2$ is the set of maximal cliques of $H(m,n)$. Adjacency is given by inclusion. The group $G=S_n\Wr S_k$ is a group of automorphisms of both $H(m,n)$ and $\Gamma$, and acts transitively on the edges of $\Gamma$ with orbits $\Delta_1$ and $\Delta_2$ on vertices.

Let $\omega\in\Omega$ and $w=(\omega,\ldots,\omega)\in\Delta_1$. Then $G_w=S_{n-1}\Wr S_k$ and the maximal cliques of $H(m,n)$ containing $w$ are $$\{(\alpha,\omega,\ldots,\omega)\mid \alpha\in \Omega\}, \{(\omega,\alpha,\omega,\ldots,\omega)\mid \alpha\in \Omega\},\,\, \ldots, \,\, \{(\omega,\ldots,\omega,\alpha)\mid \alpha\in\Omega\}$$ Hence $G_w^{\Gamma(w)}=S_k$. Moreover, letting $v=\{(\alpha,\omega,\ldots,\omega)\mid \alpha\in \Omega\}$ we have that $G_v= S_n\times (S_n\Wr S_{k-1})$ and $G_v^{\Gamma(v)}=S_n$. Thus by Lemma \ref{l:suff star}, $\Gamma$ is star-transitive when $k\neq n$. Now $G_{vw}=S_{n-1}\times (S_{n-1}\Wr S_{k-1})$ and this induces $S_{k-1}$ on $\Gamma(w)\backslash\{v\}$ and independently induces $S_{n-1}$ on $\Gamma(v)\backslash\{w\}$. Hence by Lemma \ref{lem:edgeaction}, $\Gamma$ is st(edge)-transitive when $k\neq n$. This provides an example for case (3) of Theorem \ref{vtx-intrans} with $G_v$ and $G_w$ being as large as possible.

Let $\sigma\in S_n$ such that $\omega^\sigma=\omega$ and $S_n=\langle A_n,\sigma\ra$. Then for $H=\langle A_n^k,(\sigma,\ldots,\sigma)\rangle\rtimes S_k\cong (A_n^k).2.S_k$,  we have 
$$H_w=\langle (A_{n-1})^k,(\sigma,\ldots,\sigma)\rangle \rtimes S_k\cong (A_{n-1})^k.2.S_k$$ 
$$H_v=\langle A_n\times (A_{n-1})^{k-1},(\sigma,\ldots,\sigma)\rangle\rtimes S_{k-1}\cong (A_n\times (A_{n-1})^{k-1}).2.S_{k-1}$$ and 
$$H_{vw}= \langle A_{n-1}^k,(\sigma,\ldots,\sigma)\rangle S_{k-1}$$
 Moreover, we still have that $H_w^{\Gamma(w)}\cong S_{k}$, $H_v^{\Gamma(v)}\cong S_n$ and $H_{vw}$ induces $S_{k-1}$ on $\Gamma(w)\backslash\{v\}$ and independently induces $S_{n-1}$ on $\Gamma(v)\backslash\{w\}$. Thus $\Gamma$ is also $H$-star-transitive and $H$-st(edge)-transitive. This gives an example for case (3) of Theorem \ref{vtx-intrans} with the vertex stabilisers being as small as possible.

\subsubsection{}

This example is a specialisation of \cite[Example 4.6]{GLP2}. Let $n$ be a positive integer coprime to 3, let $V=\GF(3)^n$ and $W$ be the subspace of $V$ of all vectors $(v_1,\ldots,v_n)$ such that $\sum v_i=0$.  Let $G_0=S_n\times Z$, where $Z$ is the group of scalar transformations and $V$ is the natural permutation module for $S_n$. Then $G_0$ fixes the subspace $W$. Let $v=(1,\ldots,1,-n+1)\in W$ and let $\Delta$ be the set of all translates of images of $\la v\ra$ under $G_0$, that is, $\Delta=\{\la v\ra^g +w\mid w\in W, g\in G_0\}$. Note that $|\la v\ra^{G_0}|=n$. Let $\Gamma$ be the bipartite graph with vertex set $W\cup \Delta$ and adjacency given by inclusion. Then $\Gamma$ is biregular with bivalency $\{n,3\}$ and $G=W\rtimes G_0\leqslant \Aut(\Gamma)$.

Note that the stabiliser in $G$ of the zero vector is $G_0$ and $\Gamma(0)=\la v\ra^{G_0}$. Thus $G_0^{\Gamma(0)}=S_n$. Moreover, $\Gamma(\la v\ra)=\{\lambda v\mid \lambda\in\GF(3)\}$ and $G_{\la v\ra}=(\la v\ra\rtimes Z)\times S_{n-1}$. Thus $G_{\la v\ra}^{\Gamma(\la v\ra)}=S_3$. Hence by Lemma \ref{l:suff star}, $\Gamma$ is $G$-star-transitive. Moreover, $G_{0,\la v\ra}=S_{n-1}\times Z$ acting faithfully on $\Gamma(0)\cup \Gamma(\la v\ra)$. Since $Z\cong S_2$, Lemma \ref{lem:edgeaction} implies that $\Gamma$ is st(edge)-transitive. We note that $\Gamma$ belongs to case (i) of Lemma \ref{[2]=1}.

\end{document}